\documentclass[
a4paper]
{amsart}

\usepackage[all,cmtip]{xy}
\usepackage{amssymb}
\usepackage[usenames,dvipsnames]{xcolor}
\usepackage[utf8]{inputenc}
\usepackage{tikz}
\usetikzlibrary{matrix,arrows,calc,fit,positioning}
\usepackage{tikz-cd}
\usepackage{graphicx}
\usepackage{mathrsfs}

\newtheorem{theorem}[equation]{Theorem}
\newtheorem{corollary}[equation]{Corollary}

\newtheorem{proposition}[equation]{Proposition}
\theoremstyle{definition}
\newtheorem{definition}[equation]{Definition}
\theoremstyle{remark}
\newtheorem{remark}[equation]{Remark}

\numberwithin{equation}{section}

\tikzset{mathlabel/.style={
		execute at begin node=\scriptsize $,
		execute at end node=$
	}}

\newcommand{\C}[1]{\mathscr{#1}}
\newcommand{\abs}[1]{|#1|}
\def\op{\operatorname{op}}
\def\env{\operatorname{env}}
\def\To{\longrightarrow}

\def\hom{\operatorname{Hom}}
\def\homst{\underline{\operatorname{Hom}}}

\def\st{\stackrel} 
\def\chain{\operatorname{Ch}}
\def\derived{D}
\newcommand{\hh}[2]{HH^{#1}(#2)}
\newcommand{\hbw}[2]{H^{#1}(#2)}
\newcommand{\hc}[2]{C^{#1}(#2)}
\newcommand{\hz}[2]{\breve{C}^{#1}(#2)}
\newcommand{\modulesfp}[1]{\operatorname{mod}(#1)}
\newcommand{\modulesst}[1]{\underline{\operatorname{mod}}(#1)}
\newcommand{\modules}[1]{\operatorname{Mod}(#1)}
\newcommand{\free}[1]{\mathcal{F}(#1)}

\def\ext{\operatorname{Ext}}
\def\ker{\operatorname{Ker}}
\def\coker{\operatorname{Coker}}
\def\im{\operatorname{Im}}

\begin{document}

\title
{The first obstructions to enhancing a triangulated category}%
\author{Fernando Muro}%
\address{Universidad de Sevilla,
Facultad de Matemáticas,
Departamento de Álgebra,
Avda. Reina Mercedes s/n,
41012 Sevilla, Spain}
\email{fmuro@us.es}
\urladdr{http://personal.us.es/fmuro}

\thanks{The author was partially supported by the Spanish Ministry of Economy under the grant MTM2016-76453-C2-1-P (AEI/FEDER, UE)}
\subjclass
{}
\keywords{}

\begin{abstract}
	In this paper we relate triangulated category structures to the cohomology of small categories and define initial obstructions to the existence of an algebraic or topological enhancement. We show that these obstructions do not vanish in an example of triangulated category without models. We also obtain cohomological characterizations of pre-triangulated DG, $A$-infinity, and spectral categories.
\end{abstract}

\maketitle

\tableofcontents


\section{Introduction}

Heller \cite{heller_stable_1968} noted that a triangulated structure on an essentially small additive category $\C T$ with suspension functor $\Sigma\colon\C T\rightarrow\C T$ induces a stable \emph{Toda bracket} partial composition-like operation, sending morphisms
\begin{equation}\label{three_maps}
X\stackrel{f}{\To}Y\stackrel{g}{\To}Z\stackrel{h}{\To}T
\end{equation}
with $gf=0$ and $hg=0$ to a coset
\[\langle h,g,f\rangle\subset\C T(X,\Sigma^{-1}T),\]
satisfying certain properties. Exact triangles
\[X\stackrel{f}{\To}Y\stackrel{i}{\To}C_f\stackrel{q}{\To}\Sigma X\]
are characterized by the fact that the Toda bracket contains the identity map
\[\operatorname{id}_X\in \langle q,i,f\rangle\subset\C T(X,X).\]

The graded category $\C T_{\Sigma}$ associated with the pair $(\C T, \Sigma)$ is given by
\[\C T_{\Sigma}^n(X,Y)=\C T(X,\Sigma^nY),\quad n\in\mathbb Z.\]
An \emph{algebraic enhancement} of $\C T$ in the sense of Bondal and Kapranov \cite{bondal_enhanced_1991} is a DG-category $\C C$ with $H^*(\C C)=\C T_{\Sigma}$ such that the previous Toda brackets coincide with the standard Massey products in the cohomology of $\C C$.

Assume we are working over a field $k$. By Kadeishvili's theorem \cite{kadeishvili_theory_1980, lefevre-hasegawa_sur_2003}, a Bondal--Kapranov enhancement is the same as a minimal $A$-infinity category structure on $\C T_{\Sigma}$. The first possibly non-trivial piece of this structure is a multilinear ternary composition operation $m_3$, defined on chains of three composable morphisms like \eqref{three_maps} without further conditions, such that
\[m_3(h,g,f)\in\C T(X,\Sigma^{-1}T)\]
is a well-defined element. The connection with the triangulated structure is that
\[m_3(h,g,f)\in\langle h,g,f\rangle\]
whenever the Toda bracket is defined. The compatibility properties between $m_3$ and composition in $\C T$ amount to saying that $m_3$ is a Hochschild cocycle. We denote its cohomology class by
\[\{m_3\}\in \hh{3,-1}{\C T_{\Sigma}, \C T_{\Sigma}}\]
and call it \emph{universal Massey product}. This class, previously considered in e.g.~\cite{kadeishvili_algebraic_1982, benson_realizability_2004}, is independent of the choice of minimal model.

Two natural questions arise: Does a given triangulated category have an enhancement? If so, how many essentially different ones? There are remarkable results on the existence and uniqueness of enhancements for certain triangulated categories \cite{lunts_uniqueness_2010, canonaco_uniqueness_2015} as well as examples which do not admit any enhancement \cite{muro_triangulated_2007, dimitrova_triangulated_2009, rizzardo_k-linear_2018}. There are even examples with essentially different algebraic enhancements over a field \cite{kajiura_-enhancements_2013} and some others where the first question remains open \cite{amiot_structure_2007}. In this paper we give the first step towards a different, obstruction-theoretic approach applicable to any $\C T$.

We start by considering the set of triangulated structures on a pair $(\C T,\Sigma)$ as above. For the moment we only consider triangulated structures in the sense of Puppe \cite{puppe_formal_1962}, i.e.~not requiring Verdier's octahedral axiom. Freyd \cite{freyd_stable_1966} proved that, for that set to be non-empty, the category $\modulesfp{\C T}$ of finitely presented right $\C T$-modules must be a Frobenius abelian category. In that case we can define the stable module category $\modulesst{\C T}$, which is canonically triangulated with suspension functor $S$, the cosyzygy functor, and $\Sigma$ induces a triangulated endofunctor of $\modulesst{\C T}$. If in addition idempotents split in $\C T$, Heller \cite{heller_stable_1968} defined a bijection between the set of triangulated structures on $(\C T,\Sigma)$ and a subset of the set of natural transformations $\Sigma\rightarrow S^3$ between endofunctors of $\modulesst{\C T}$ satisfying two algebraic conditions. 
The requirement on idempotents is harmless because any triangulated structure extends uniquely to the idempotent completion of $\C T$ \cite{balmer_idempotent_2001}.

In Section \ref{heller_section} we identify the set of triangulated structures with a subset of the Hochschild cohomology group
\[\hh{0,-1}{\modulesfp{\C T_{\Sigma}},\ext_{\C T_{\Sigma}}^{3,\ast}},\]
where $\modulesfp{\C T_{\Sigma}}$ is the graded abelian category of finitely presented right $\C T_{\Sigma}$-modules. This group is usually strictly smaller than Heller's set of natural transformations $\Sigma\rightarrow S^3$. The subset of triangulated structures is again defined by two algebraic conditions on the Hochschild cohomology classes which correspond to Heller's. The second condition is identified later, in Section \ref{spectral_sequence_section} for triangulated categories over a field, and in Section \ref{topological_section} in general.

We prove in Section \ref{toda_brackets_section} that the previous Hochschild cohomology group is in bijection with the set of all stable Toda brackets in $(\C T,\Sigma)$. We also show that unstable Toda brackets are in bijection with another (ungraded) Hochschild cohomology group where the former injects. The latter is also in bijection with the set of natural transformations $\Sigma\rightarrow S^3$ used by Heller. This establishes a direct link between Toda brackets and natural transformations $\Sigma\rightarrow S^3$ not considered by Heller in \cite{heller_stable_1968} despite he used both.

In Sections \ref{spectral_sequence_section} and \ref{octahedral_axiom_section} we restrict ourselves to working over a ground field $k$. In the first one we define a first quadrant spectral sequence of graded vector spaces
\[E_2^{p,q}=\hh{p,\ast}{\modulesfp{\C T_{\Sigma}},\ext_{\C T_{\Sigma}}^{q,\ast}}\Longrightarrow\hh{p+q,\ast}{\C T_{\Sigma}, \C T_{\Sigma}}\]
whose edge morphism
\begin{equation}\label{edge}
\hh{3,-1}{\C T_{\Sigma}, \C T_{\Sigma}}\To \hh{0,-1}{\modulesfp{\C T_{\Sigma}},\ext_{\C T_{\Sigma}}^{3,\ast}}
\end{equation}
takes the universal Massey product of any enhancement to the stable Toda bracket of the triangulated structure,
\[\{m_3\}\mapsto \langle-,-,-\rangle.\]
The spectral sequence is not totally new, a related ungraded version has been considered in \cite{lowen_hochschild_2005}.

Using the connection between the stable Toda bracket of a triangulated category and the universal Massey product of any possible enhancement via the previous spectral sequence, we can define the first obstructions to the existence of a Bondal--Kapranov enhancement. The very first obstruction is the image of the Toda bracket along the spectral sequence differential
\begin{equation}\label{d2}
d_2\colon \hh{0,-1}{\modulesfp{\C T_{\Sigma}},\ext_{\C T_{\Sigma}}^{3,\ast}}\To \hh{2,-1}{\modulesfp{\C T_{\Sigma}},\ext_{\C T_{\Sigma}}^{2,\ast}}.
\end{equation}
If it vanishes, then $\langle-,-,-\rangle\in E_3^{0,3}\subset \hh{0,\ast}{\modulesfp{\C T_{\Sigma}},\ext_{\C T_{\Sigma}}^{3,\ast}}$ is in the third page of the spectral sequence and the second obstruction is its image along
\[d_3\colon E_3^{0,3}\To E^{3,1}_3.\]
The second obstruction vanishes when $\langle-,-,-\rangle\in E_4^{0,3}\subset E_3^{0,3}$ is in the fourth page. In this case there is a third obstruction, the image of the Toda bracket under
\[d_4\colon E_4^{0,3}\To E^{4,0}_4.\]
When it vanishes, $\langle-,-,-\rangle\in E_\infty^{0,3}=E_5^{0,3}\subset E_4^{0,3}$ is a permament cycle, i.e.~the Toda bracket is in the image of the edge morphism \eqref{edge}. If this happens, any preimage is a potential universal Massey product, that is to say, if a representing cocycle $m_3$ can be extended to a full minimal $A$-infinity algebra structure on $\C T_{\Sigma}$, then this extension is an enhancement for the triangulated structure on $\C T$. Abusing terminology, we will sometimes say that a triangulated category over a field has a universal Massey product if its stable Toda bracket has a preimage along the edge morphism \eqref{edge}.

There is a well-known classical obstruction theory with values in Hochschild cohomology for the extension of truncated minimal $A$-infinity category structures, see e.g.~\cite{lefevre-hasegawa_sur_2003}. We have developed an enhanced version in \cite{muro_enhanced_2015}, after Angeltveit \cite{angeltveit_topological_2008}. It will be further investigated in the triangulated context in a subsequent paper.

Any Puppe triangulated category with an enhacement satisfies Verdier's octahedral axiom. However, it does not seem to be possible to translate this axiom into algebraic conditions on Hochschild cohomology classes. In Section \ref{octahedral_axiom_section} we give a sufficient algebraic condition for the octahedral axiom: it holds provided the Toda bracket classifying the triangulated structure is in the kernel of the spectral sequence differential \eqref{d2}. In particular, any triangulated structure over a field with a universal Massey product satisfies the octahedral axiom. This also leads to an apparently new homological characterization of pre-triangulated DG- and $A$-infinity categories over a field in the sense of \cite{bondal_enhanced_1991}.

Many triangulated categories do not have algebraic enhancements. The most prominent example is the stable homotopy category of spectra. Nevertheless, most triangulated categories have a \emph{topological enhancement}, which consists of a full embedding into the homotopy category $\operatorname{Ho}\C M$ of a stable model category $\C M$ \cite{schwede_algebraic_2010,hovey_model_1999}. In the topological context, enhancements are usually called \emph{models}. Algebraic enhancements are also topological.

Using \cite{baues_homotopy_2007}, we show in Section \ref{topological_section}  that a topological enhancement of $\C T$ gives rise to a cohomology class
\[\{m_3\}\in \hbw{3,-1}{\C T_{\Sigma}, \C T_{\Sigma}}\]
in the non-additive cohomology of the category $\C T_{\Sigma}$, which is a generalization of Mac Lane and topological Hochschild cohomology of rings. This cohomology class is called \emph{universal Toda bracket}. Formally, a cocycle $m_3$ representing the universal Toda bracket is a ternary operation as above, except for the fact that it need not be multiliear. The connection to stable Toda brackets in the triangulated category $\C T$ is exactly as above. Universal Toda brackets were introduced in \cite{baues_cohomology_1989}, where their connection with ordinary Toda brackets is also established.

Hochschild cohomology coincides with its non-additive counterpart in dimension $0$, in particular 
\[\hh{0,-1}{\modulesfp{\C T_{\Sigma}},\ext_{\C T_{\Sigma}}^{3,\ast}}\cong \hbw{0,-1}{\modulesfp{\C T_{\Sigma}},\ext_{\C T_{\Sigma}}^{3,\ast}},\]
so both of them are in bijection with stable Toda brackets in $(\C T,\Sigma)$.

In Section \ref{topological_section} we define a first quadrant spectral sequence of graded abelian groups for non-additive cohomology of categories
\[E_2^{p,q}=\hbw{p,\ast}{\modulesfp{\C T_{\Sigma}},\ext_{\C T_{\Sigma}}^{q,\ast}}\Longrightarrow\hbw{p+q,\ast}{\C T_{\Sigma}, \C T_{\Sigma}}\]
whose edge morphism
\begin{equation}\label{edge_top}
\hbw{3,-1}{\C T_{\Sigma}, \C T_{\Sigma}}\To \hbw{0,-1}{\modulesfp{\C T_{\Sigma}},\ext_{\C T_{\Sigma}}^{3,\ast}}
\end{equation}
takes the universal Toda bracket of any topological enhancement to the stable Toda bracket of the triangulated structure,
\[\{m_3\}\mapsto \langle-,-,-\rangle.\]
An ungraded version of this spectral sequence appears in \cite{jibladze_cohomology_1991} with an apparently different target. 

We can define the first three obstructions to the existence of a topological enhancement as above. We also show that the previous sufficient condition for Verdier's octahedral axiom is still valid for this spectral sequence. In particular any triangulated category with a universal Toda bracket satisfies the octahedral axiom. This leads to an new homological characterization of triangulated spectral categories in the sense of Tabuada \cite{tabuada_matrix_2010}. Moreover, it shows that, if we ever find a Puppe triangulated category which does not satisfy the octahedral axiom, then the very first obstruction to the existence of a topological enhancement must be non-zero.

We conclude in Section \ref{example} with an explicit example where the obstructions do not vanish. The triangulated category $\C T$ in the example is the category of finitely generated free modules over $\mathbb Z/4$ with $\Sigma$ the identity functor. This category is among the first known examples of triangulated categories without models of any kind \cite{muro_triangulated_2007}.

\section{Hochschild cohomology of categories}\label{Hochschild_cohomology_of_categories}

We work over a ground commutative ring $k$ and graded objects are $\mathbb Z$-graded. The degree of $x$ is denoted by $\abs{x}$. Let $\modules{k}$
be the category of graded $k$-modules equipped with the usual closed symmetric monoidal structure, where the symmetry constraint uses the Koszul sign rule. The tensor product will be denoted by $\otimes$ and the inner $\hom$ by $\hom_k^*$. Since we will not change rings, $k$ will often be dropped from notation.

A \emph{graded $k$-linear category}, or just \emph{graded category}, is a category enriched in $\modules{k}$, and similarly graded functors and graded natural transformations. Note that graded natural transformations may have any degree $n\in\mathbb Z$, while graded functors must be defined by degree $0$ morphisms between their graded modules of morphisms. We refer to \cite{kelly_basic_2005} for enriched category concepts. For the time being, all categories will be linear, so we will often drop this word.

Given a small graded category $\C C$, a \emph{right $\C C$-module} $M$ is a graded functor \[M\colon \C C^{\op}\To \modules{k}\] from the opposite graded category $\C C^{\op}$ to the category $\modules{k}$ of graded $k$-modules, i.e.~for each object $X$ in $\C C$ a graded module $M(X)$ is given, and for each pair of objects $X,X'$ in $\C C$ a degree $0$ morphism of graded modules
\begin{align*}
M(X)\otimes\C C(X',X)&\longrightarrow M(X'),\\
x\otimes f&\;\mapsto\; x\cdot f,
\end{align*}
satisfying the obvious associativity and unit conditions. A \emph{morphism} of right $\C C$-modules $g\colon M\rightarrow N$ of any degree $n\in\mathbb Z$ is a collection of degree $n$ morphims of graded $k$-modules $g(X)\colon M(X)\rightarrow N(X)$, $X$ an object in $\C C$, such that, with the notation above, $g(X')(x\cdot f)=g(X)(x)\cdot f$.

Right $\C C$-modules form a graded abelian category $\modules{\C C}$. Representable functors $\C C(-,X)$ and their shifts form a set of projective generators. A right $\C C$-module is \emph{finitely presented} if it is the cokernel of a morphism between finite direct sums of these projective generators. The full subcategory of finitely presented right $\C C$-modules will be denoted by $\modulesfp{\C C}$. Morphism graded $k$-modules in $\modules{\C C}$ are denoted by $\hom_{\C C}^\ast$ and their derived functors by $\ext^{n,\ast}_{\C C}$, $n\geq 0$. A good reference for this kind of graded categorical algebra is Street's thesis \cite{street_homotopy_1969}. 

A \emph{bimodule} $M$ over $\C C$ is a right $\C C^{\env}$-module, where $\C C^{\env}=\C C\otimes \C C^{\op}$, i.e.~a graded functor
\[M\colon \C C^{\env}=\C C^{\op}\otimes\C C\To \modules{k}.\]
Equivalently, $M$ is a family of graded modules $M(X,Y)$ indexed by pairs of objects $X,Y$ in $\C C$
and, for each four objects $X,X',Y,Y'$ in $\C C$, a graded module morphism of degree $0$
\begin{align*}
\C C(Y,Y')\otimes M(X,Y)\otimes \C C(X',X)&\longrightarrow M(X',Y'),\\
g\otimes x\otimes f&\;\mapsto\; g\cdot x\cdot f,
\end{align*}
satisfying the usual associativity and unit conditions. The $\C C$-bimodule $\C C(-,-)$ will be simply denoted by $\C C$. A \emph{morphism} of $\C C$-bimodules $h\colon M\rightarrow N$ is a family of morphims of graded $k$-modules $h(X,Y)\colon M(X,Y)\rightarrow N(X,Y)$ satisfying \[h(X',Y')(g\cdot x\cdot f)=(-1)^{\abs{g}\abs{h}}g\cdot h(X,Y)(x)\cdot f.\]

The \emph{bar complex} $B_\star(\C C)$ is the chain complex of $\C C$-bimodules concentrated in dimensions $\geq 0$ defined by
\[B_n(\C C)=\bigoplus_{X_0,\dots, X_n} \C C(X_0,-)\otimes\C C(X_1,X_{0})\otimes\cdots\otimes\C C(X_n,X_{n-1})\otimes\C C(-,X_n),\]
where the coproduct is indexed by the sequences of $n+1$ objects of $\C C$,
with differential
\[d(f_0\otimes\dots\otimes f_{n+1})=\sum_{i=0}^n(-1)^i\cdots\otimes f_if_{i+1}\otimes\cdots.\]
It is a relative projective resolution of $\C C$ (with respect to $k$-split surjections) with augmentation 
\[\epsilon\colon B_\star(\C C)\To\C C,\qquad\epsilon(f_0\otimes f_1)=f_0f_1.\]
If $\C C$ is locally projective, i.e.~if morphism graded modules in $\C C$ are projective, then it is an honest resolution. 
We refer to \cite{mac_lane_homology_1963} for basic relative homological algebra. 

The \emph{Hochschild cohomology} $\hh{\star,*}{\C C,M}$ of $\C C$ with coefficients in a $\C C$-bimodule $M$ is the cohomology of this complex (we should maybe say \emph{Hochschild--Mitchell cohomology} \cite{mitchell_rings_1972}). Using adjunction and the Yoneda lemma, the \emph{Hochschild cochain complex} $\hc{\star,*}{\C C,M}=\hom_{\C C^{\env}}^*(B_\star(\C C),M)$ is given in Hochschild degree $\star=n$ by
\[\hc{n,*}{\C C,M}=\prod_{X_0,\dots, X_n} \hom_{k}^*(\C C(X_1,X_{0})\otimes\cdots\otimes\C C(X_n,X_{n-1}),M(X_n,X_0)),\]
and the differential is given by
\begin{multline}\label{hochschild_differential}
d(\varphi)(f_1,\dots,f_{n+1})={}(-1)^{\abs{\varphi}\abs{f_1}}f_1\cdot \varphi(f_2,\dots,f_{n+1})\\+\sum_{i=1}^n(-1)^i \varphi(\dots, f_if_{i+1},\dots)+(-1)^{n+1}\varphi(f_1,\dots,f_n)\cdot f_{n+1}.
\end{multline}
If $M$ is a monoid in the category of $\C C$-bimoules, then Hochschild cohomology becomes a bigraded ring with the well-known cup-product. 
Note that $\hh{0,*}{\C C,M}$ coincides in general with the \emph{end} of the bifunctor $M$.

Hochschild cohomology is functorial in both $\C C$ and $M$. A graded functor $F\colon \C D\rightarrow \C C$ and a $\C C$-bimodule morphism $\varphi\colon N\rightarrow M$ induce morphisms
\begin{align*}
F^*\colon \hh{\star,\ast}{\C C,M}&\To \hh{\star,\ast}{\C D,M(F,F)},\\
\varphi_*\colon \hh{\star,\ast}{\C C,N}&\To \hh{\star,\ast+\abs{\varphi}}{\C C,M}.
\end{align*}
They satisfy $F^*\varphi_*=\varphi(F,F)_*F^*$. 
The bivariant functoriality can be described as in \cite{muro_functoriality_2006}, in particular categorical equivalences induce isomorphisms. 

We consider the ungraded case as the particular instance of the former where everything is concentrated in degree $0$. Over ungraded categories, we can consider both graded and ungraded (bi)modules. If $\C T$ is an ungraded category and $M$ is a graded $\C T$-bimodule, then the cohomology of $\C T$ with coefficients in $M$ reduces to ungraded cohomology, \[\hh{p,q}{\C T,M}=\hh{p}{\C T,M^q}.\] 
We now relate the cohomology of graded and ungraded categories.

We have sketched in the introduction how out of an arbitrary ungraded category $\C T$ equipped with an automorphism $\Sigma$, we can form a graded category $\C T_{\Sigma}$. 
Composition in $\C T_\Sigma$ is defined as follows,
\begin{align*}
\C T_{\Sigma}^p(Y,Z)\otimes \C T_{\Sigma}^q(X,Y)&\To\C T_\Sigma^{p+q}(X,Z),\\ 
f\otimes g&\;\mapsto\; (\Sigma^qf)g.
\end{align*}
Here, on the right, we have a composition in $\C T$. 
We can extend $\Sigma$ to an automorphism $\Sigma \colon\C T_\Sigma\rightarrow\C T_\Sigma$, defined as in $\C T$ on objects, and on morphisms as $(-1)^n\Sigma$ in each degree $n\in\mathbb Z$. In this way, the graded extension $\Sigma$ is equipped with a natural isomorphism 
\begin{equation}\label{imath}
\imath_X\colon X\cong \Sigma X
\end{equation} of degree $-1$ given by the identity in $X$. The sign in the extension of $\Sigma$ is necessary for the graded naturality, because of Koszul's sign rule. Graded categories equivalent to some $\C T_\Sigma$ are called \emph{weakly stable} \cite{street_homotopy_1969}. They are characterized by the fact that shifts of representable functors are representable, or equivalently, each object has an isomorphism of any given degree.

Similarly, if $M$ is an ungraded $\C T$-bimodule equipped with an isomorphism
$\tau\colon M\cong M(\Sigma,\Sigma)$ such that, given
$g\in\C T(Y,Y')$, $x\in M(X,Y)$, and $f\in\C T(X',X)$,
\[\tau(g\cdot x\cdot f)=(\Sigma g)\cdot\tau(x)\cdot(\Sigma f),\]
then we can form a graded $\C T_{\Sigma}$-bimodule $M_\tau$ defined by
\[M_\tau^n(X,Y)=M(X,\Sigma^nY).\]
The bimodule structure is defined as
\begin{align*}
\C T^p_\Sigma(Y,Y')\otimes M^q_\tau(X,Y)\otimes \C T^r_\Sigma (X',X)
&\To
M^{p+q+r}_\tau(X',Y'),\\
g\otimes x\otimes f&\;\mapsto\;(\Sigma^{q+r}g)\cdot (\tau^rx)\cdot f.
\end{align*}
Here, on the right, we use the suspension in $\C T$ and the $\C T$-bimodule structure of $M$ (no signs involved). Moreover, we can extend $\tau$ to a degree $0$ isomorphism of $\C T_\Sigma$-bimodules $\tau\colon M_\tau\cong M_\tau(\Sigma,\Sigma)$ defined as $(-1)^n\tau$ in each degree $n\in\mathbb Z$. The sign in the definition of $\tau$ is needed to cancel the signs in the $\C T_\Sigma$-bimodule morphism equation for $\tau$ arising from the definition of the suspension in $\C T_\Sigma$. 

Since $\C T\subset \C T_{\Sigma}$ is the degree $0$ part, we can also regard $M_\tau$ as a graded $\C T$-bimodule.

\begin{proposition}\label{graded_ungraded_long_exact_sequence}
	If $\C T$ is an ungraded category equipped with an automorphism $\Sigma\colon \C T\rightarrow\C T$, $M$ is an ungraded $\C T$-bimodule equipped with an isomorphism $\tau\colon M\cong M(\Sigma,\Sigma)$ satisfying $\tau(f\cdot x\cdot g)=(\Sigma f)\cdot\tau(x)\cdot(\Sigma g)$, and $i\colon\C T\subset\C T_{\Sigma}$ denotes the inclusion of the degree $0$ part, then there is a long exact sequence
	\begin{center}
		\begin{tikzcd}[row sep=5mm]
			\vdots\arrow[d]\\
			\hh{n,*}{\C T_{\Sigma},M_{\tau}}\arrow[d, "i^*"]\\
			\hh{n,*}{\C T,M_{\tau}}\arrow[d, "\operatorname{id}-\tau_*^{-1}\Sigma^*"]\\ 
			\hh{n,*}{\C T,M_{\tau}}\arrow[d]\\
			\hh{n+1,*}{\C T_{\Sigma},M_{\tau}}\arrow[d]\\
			\vdots
		\end{tikzcd}
	\end{center}
\end{proposition}

\begin{proof}
	We consider the $\C T_\Sigma$-bimodule $\C T_\Sigma\otimes_{\C T}\C T_\Sigma$. It is degreewise given by
	\begin{equation*}
	\begin{split}
		(\C T_\Sigma\otimes_{\C T}\C T_\Sigma)^n&= \bigoplus_{p+q=n}\C T_\Sigma^p\otimes_{\C T}\C T_\Sigma^q\\
		&=\bigoplus_{p+q=n}\C T(-,\Sigma^p)\otimes_{\C T}\C T(-,\Sigma^q)\\
		&=\bigoplus_{q\in\mathbb Z}\C T(-,\Sigma^{n-q})\otimes_{\C T}\C T(-,\Sigma^q)\\
		{\scriptstyle (\bigoplus_{q\in\mathbb Z}\Sigma^q\otimes 1)}&\cong
		\bigoplus_{q\in\mathbb Z}\C T(\Sigma^q,\Sigma^n)\otimes_{\C T}\C T(-,\Sigma^q)\\
		{\scriptstyle (\text{composition})}&\cong \bigoplus_{q\in\mathbb Z}\C T(-,\Sigma^n).
	\end{split}
	\end{equation*}
	The natural transformation \eqref{imath} defines a degree $0$ bimodule automorphism
	\[\Gamma\colon \C T_\Sigma\otimes_{\C T}\C T_\Sigma\cong \C T_\Sigma\otimes_{\C T}\C T_\Sigma\colon
	f\otimes g\mapsto f\imath^{-1}\otimes\imath g.\]
	Indeed, the left (resp.~right) tensor coordinate shifts its degree by $+1$ (resp.~$-1$) so, as a whole, it has degree $0$. 
	Degreewise, $\Gamma$ is the automorphism of $\bigoplus_{q\in\mathbb Z}\C T(-,\Sigma^n)$ which shifts coordinates one step downwards. Therefore, the degree $0$ bimodule morphism
	\[\operatorname{id}-\Gamma\colon \C T_\Sigma\otimes_{\C T}\C T_\Sigma\longrightarrow \C T_\Sigma\otimes_{\C T}\C T_\Sigma \]
	is injective (its kernel would be, degreewise, the elements whose coordinates are all equal to the previous coordinate, hence zero since we are in a direct sum).
	Moreover, the bimodule morphism defined by composition 
	\[\C T_\Sigma\otimes_{\C T}\C T_\Sigma\longrightarrow \C T_\Sigma\colon f\otimes g\mapsto fg\]
	is degrewise given by the identity in $\C T(-,\Sigma^n)$ on each direct summand, so it is the cokernel of $\operatorname{id}-\Gamma$, since $\operatorname{id}-\Gamma$ is actually the standard presentation of the colimit of the $\mathbb Z$-indexed diagram given by the identity in $\C T(-,\Sigma^n)$ everywhere. Suming up, we have a short exact sequence of $\C T_\Sigma$-bimodules
	\[\C T_\Sigma\otimes_{\C T}\C T_\Sigma\stackrel{\operatorname{id}-\Gamma}\hookrightarrow \C T_\Sigma\otimes_{\C T}\C T_\Sigma\stackrel{\text{comp.}}\twoheadrightarrow \C T_\Sigma.\]
	We will use it to construct a convenient $\C T_\Sigma$-bimodule resolution of $\C T_\Sigma$ from a resolution of $\C T_\Sigma\otimes_{\C T}\C T_\Sigma$ and a lift of $\operatorname{id}-\Gamma$.
		
	The extension of scalars of the bar complex of $\C T$-bimodules $B_*(\C T)$ along the inclusion $i$ is $\C T_{\Sigma}\otimes_{\C T} B_*(\C T)\otimes_{\C T} \C T_{\Sigma}$, which is a relative projective resolution of $\C T_{\Sigma}\otimes_{\C T} \C T_{\Sigma}$. At each bar degree it is given by
	\begin{multline*}
		\C T_{\Sigma}\otimes_{\C T} B_n(\C T)\otimes_{\C T} \C T_{\Sigma}\\=\bigoplus_{X_0,\dots, X_n} \C T_{\Sigma}(X_0,-)\otimes\C T(X_1,X_{0})\otimes\cdots\otimes\C T(X_n,X_{n-1})\otimes\C T_\Sigma(-,X_n).
	\end{multline*}
	By adjunction, we can use this complex to compute the cohomology of $\C T$ with coefficients in the restriction of a $\C T_\Sigma$-bimodule (e.g.~$M_\tau$) along $i$.
	
	The automorphism $\Gamma$ lifts to the bar resolution 
	\[\Gamma\colon \C T_{\Sigma}\otimes_{\C T} B_*(\C T)\otimes_{\C T} \C T_{\Sigma}\cong \C T_{\Sigma}\otimes_{\C T} B_*(\C T)\otimes_{\C T} \C T_{\Sigma}\]
	by means of the $\C T_\Sigma$-bimodule isomorphisms
	\begin{multline*}
	\C T_{\Sigma} (X_0,-)\otimes\C T(X_1,X_{0})\otimes\cdots\otimes\C T(X_n,X_{n-1})\otimes\C T_{\Sigma} (-,X_n)\\\cong 
	\C T_{\Sigma} (\Sigma X_0,-)\otimes\C T(\Sigma X_1,\Sigma X_{0})\otimes\cdots\otimes\C T(\Sigma X_n,\Sigma X_{n-1})\otimes\C T_{\Sigma} (-,\Sigma X_n)
	\end{multline*}
	defined by 
	\begin{align*}
	f_0\otimes f_1\otimes\cdots \otimes f_n\otimes f_{n+1}&\mapsto f_0\imath^{-1} \otimes (\Sigma f_1)\otimes\cdots\otimes(\Sigma f_n)\otimes \imath f_{n+1}.
	\end{align*}
	The mapping cone of $\operatorname{id}-\Gamma$ is therefore a relative projective resolution of $\C T_\Sigma$ as a bimodule over itself, which can be used to compute the cohomology of this category. The exact sequence in the statement will be the long exact cohomology sequence of the standard exact triangle completion of $\operatorname{id}-\Gamma$ with coefficients in $M_\tau$. It is only left to identify the morphism induced by $\Gamma$ on cohomology with $\tau^{-1}_*\Sigma^*$. We do this in the following  paragraph, actually at the level of cochains.
	
	A $\C T_\Sigma$-bimodule morphism of degree $m$
	\[\varphi\colon \C T_{\Sigma}\otimes_{\C T} B_n(\C T)\otimes_{\C T} \C T_{\Sigma}\longrightarrow M_{\tau}\]
	identifies with a collection of $k$-module morphisms
	\[\varphi\colon \C T(X_1,X_{0})\otimes\cdots\otimes\C T(X_n,X_{n-1})\longrightarrow M_\tau^m(X_n,X_0)=M(X_n,\Sigma^mX_0),\]
	one for each sequence of objects $X_0,\dots, X_n$. The composite $\varphi\Gamma$ is given by
	\begin{equation*}
	\begin{split}
	\varphi\Gamma(f_1\otimes\cdots \otimes f_n)&=(-1)^m\imath^{-1}_{\Sigma^mX_0}\varphi((\Sigma f_1)\otimes\cdots\otimes(\Sigma f_n))\imath_{X_n}\\
	&=(-1)^m\tau^{|\imath_{X_n}|}\varphi((\Sigma f_1)\otimes\cdots\otimes(\Sigma f_n))\\
	&=\tau^{-1}\varphi(\Sigma^{\otimes^n})(f_1\otimes\cdots \otimes f_n).
	\end{split}
	\end{equation*}
	Here we use that $\imath$ has deegree $-1$ and the $f_i$'s have degree $0$, that $\imath$ is given by identity maps, and the definitions of the $\C T_\Sigma$-module $M_\tau$ and the graded $\tau$.
\end{proof}

\begin{remark}[Multiplicative properties]\label{multiplicative_properties}
	In the context of the previous proposition, assume that $M$ is a monoid in the category of $\C T$-bimodules and $\tau$ is a monoid morphism. Then $M_\tau$ is a monoid in the category of $\C T_\Sigma$-bimodules with composition law defined as 
	\[M_{\tau}^p(Y,Z)\otimes M_{\tau}^q(X,Y)\To M_{\tau}^{p+q}(X,Z)\colon x\otimes y\mapsto \tau^q(x)y.\]
	Here, on the right, we use the ungraded $\tau$ and composition in $M$. The graded $\tau$ defined above becomes automatically a morphism of monoids in $\C T_\Sigma$-bimodules. Therefore, not only $\Sigma^*$ but also $\tau_*$ is a bigraded ring morphism $\hh{\star,*}{\C T,M_{\tau}}\rightarrow \hh{\star,*}{\C T,M_{\tau}(\Sigma,\Sigma)}$. In particular, $\tau_*^{-1}\Sigma^*$ is a bigraded ring endomorphism of $\hh{\star,*}{\C T,M_{\tau}}$ and the kernel of $\operatorname{id}-\tau_*^{-1}\Sigma^*$ is a bigraded subring, since it coincides with the equalizer of $\tau_*^{-1}\Sigma^*$ and the identity map.
\end{remark}

\section{Heller's classification of triangulated structures}\label{heller_section}

Triangulated categories were introduced by Puppe \cite{puppe_formal_1962} and Verdier \cite{verdier_categories_1996} at about the same time (Verdier's thesis, although widely circulated, was only published three decades later). Puppe, however, did not consider Verdier's \emph{octahedral axiom}. Eventually, Verdier's axiomatic became standard, hence we refer to Puppe triangulated structures if we do not explicitly require Verdier's additional axiom. It is remarkable that, to this day, after failed attempts, no example of Puppe triangulated category is known where the octahedral axiom fails, although it is a common belief that the octahedral axiom does not follow from the rest.

Freyd \cite{freyd_stable_1966} showed that, if an additive category $\C T$ is admits a Puppe triangulated structure with suspension $\Sigma$, then the category of finitely presented \emph{ungraded} right $\C T$-modules $\modulesfp{\C T}$ is a \emph{Frobenius} abelian category, i.e.~it has enough projectives and injectives and both classes of objects coincide (it is the class of direct summands of representable functors). 

Assume that idemponents split in $\C T$ and $\modulesfp{\C T}$ is Frobenius abelian. Heller \cite{heller_stable_1968} classified the set of Puppe triangulated structures on $\C T$ with suspension $\Sigma$ in the following way. The \emph{stable module category} $\modulesst{\C T}$ is the quotient of $\modulesfp{\C T}$ by the ideal of morphisms factoring through a representable. Morphism sets in this category are denoted by $\homst_{\C T}$. The stable module category is a triangulated category, its suspension functor is the \emph{cosyzygy} functor $S $, defined on objects by the choice of short exact sequences in $\modulesfp{\C T}$ of the form
\[M\stackrel{j_M}\hookrightarrow\C T(-,X_M)\stackrel{q_M}\twoheadrightarrow S M.\]
Given a morphism $\{f\}\colon M\rightarrow N$ in $\modulesst{\C T}$ represented by $f$ in $\modulesfp{\C T}$, the morphism $S \{f\}$ is represented by any map $S f$ fitting in a commutative diagram
\begin{center}
\begin{tikzcd}
M \arrow[r, "j_M", hook] \arrow[d, "f"'] & \C T(-,X_M) \arrow[r, "p_M", two heads] \arrow[d] & S M \arrow[d, "S f"]\\
N \arrow[r, "j_N", hook] & \C T(-,X_N) \arrow[r, "p_N", two heads] & S N
\end{tikzcd}
\end{center}
(The class of exact triangles in $\modulesst{\C T}$ is irrelevant for our purposes.) 
Moreover, the functor $\Sigma$ extends in an essentially unique way to an exact automorphism of $\modulesfp{\C T}$ through the Yoneda inclusion. Furthermore, it passes to the quotient $\modulesst{\C T}$ as a triangulated functor, part of which is a natural isomorphism \[\sigma\colon \Sigma S \cong S \Sigma\] defined by the choice of commutative diagrams in $\modulesfp{\C T}$ of the form
\begin{center}
	\begin{tikzcd}
		\Sigma M \arrow[r, "\Sigma j_M", hook] \arrow[d, equal] & \C T(-,\Sigma X_M) \arrow[r, "\Sigma p_M", two heads] \arrow[d] & \Sigma S M \arrow[d]\\
		\Sigma M \arrow[r, "j_{\Sigma M}", hook] & \C T(-,X_{\Sigma M}) \arrow[r, "p_{\Sigma M}", two heads] & S \Sigma M
	\end{tikzcd}
\end{center}

Heller \cite[Theorem 16.4]{heller_stable_1968} showed that the (possibly empty) set of Puppe triangulated structures on $\C T$ with suspension functor $\Sigma$ is in bijection with the set of natural isomorphisms \[\delta\colon\Sigma\cong S^{3}\] which anticommute with $S $, i.e.
\[(S \delta)\sigma+\delta S=0.\]

The set of natural transformations $\Sigma\rightarrow S^{3}$ is the cohomology group
\[\hh{0}{\modulesst{\C T},\homst_{\C T}(\Sigma,S^{3})}\]
since $0$-dimensional Hochschild cohomology computes the end of the coefficient bimodule. Moreover, if we consider the natural isomorphism
\[\tau\colon\modulesst{\C T}(\Sigma,S^{3})\To
\modulesst{\C T}(\Sigma S ,S^{4})\colon\delta\mapsto -(S\delta)\sigma,\]
the anticommutativity condition is equivalent to being in the kernel of 
\begin{equation}\label{anticommutativity_cohomological}
\operatorname{id}-\tau_*^{-1}S^*\colon \hh{0}{\modulesst{\C T},\homst_{\C T}(\Sigma,S^{3})}\longrightarrow
\hh{0}{\modulesst{\C T},\homst_{\C T}(\Sigma,S^{3})}.
\end{equation}

The pull-back of the $\modulesst{\C T}$-bimodule $\homst_{\C T}(-,S^n)$ along the natural projection $\modulesfp{\C T}\twoheadrightarrow\modulesst{\C T}$ is \emph{Tate}'s $\widehat{\ext}^{n}_{\C T}$, $n\in\mathbb Z$, computed by using a complete resolution of any variable instead of just a projective or an injective resolution. It coincides with the ordinary $\ext_{\C T}^n$ for $n>0$. Composition in $\modulesst{\C T}_S$ extends the Yoneda product in $\ext_{\C T}^*$. In order to compute the coend of a $\modulesst{\C T}$-bimodule, we can equally pull it back to $\modulesfp{\C T}$. Hence, combining the previous observations, we obtain an isomorphism
\[\hh{0}{\modulesst{\C T},\homst_{\C T}(\Sigma,S^{3})}
\cong
\hh{0}{\modulesfp{\C T},\widehat{\ext}_{\C T}^3(-,\Sigma^{-1})}
.\]
Any cocycle $\varphi$ on the right gives rise to a class of exact triangles. More precisely, a diagram in $\C T$ of the form 
\[X\stackrel{f}{\To}Y\stackrel{i}{\To}C_f\stackrel{q}{\To}\Sigma X\]
is a \emph{$\varphi$-exact triangle} if
\begin{center}
\begin{tikzcd}
	\Sigma^{-1}M\arrow[r,"\alpha",hook] & \C T(-,X)\arrow[r,"{\C T(-,f)}"] &\C T(-,Y)\arrow[r, "{\C T(-,i)}"]&\C T(-,C_f)\arrow[r,"\lambda",two heads]&M,
\end{tikzcd}
\end{center}
where $M=\operatorname{Im} \C T(-,q)$  and 	
$\C T(-,q)=(\Sigma\alpha)\lambda$, is an extension representing $\varphi(M)\in\widehat{\ext}_{\C T}^3(M,\Sigma^{-1}M)$. Heller's result is equivalent to saying that this defines a bijection between Puppe triangulated stuctures in $\C T$ with suspension $\Sigma$ and Hochschild $0$-cocycles satisfying the invertibility and the anticommutativity conditions. So far, in order to check these conditions we need the natural transformation $\delta$. We will now give a cohomological characterization of the invertibility condition. The anticommutativity condition has already a characterization using the cohomology of $\modulesst{\C T}$, see \eqref{anticommutativity_cohomological}. We will later give a characterization in terms of the cohomology of $\modulesfp{\C T_\Sigma}$, in Section \ref{spectral_sequence_section} over a field, and in Section \ref{topological_section} in general.

Since $\modulesfp{\C T}$ is Frobenius abelian, then $\modulesfp{\C T_{\Sigma}}$ is graded Frobenius abelian. (The former is the degree $0$ part of the latter, and these properties are characterized in degree $0$.) Therefore, on finitely presented $\C T_{\Sigma}$-modules, we have Tate's $\widehat{\ext}_{\C T_{\Sigma}}^{n,*}$, $n\in\mathbb Z$, as above. 
The invertible functor $\Sigma\colon\modulesfp{\C T}\rightarrow\modulesfp{\C T}$ induces natural isomorphisms
\[\Sigma\colon \widehat{\ext}_{\C T}^{n}\cong \widehat{\ext}_{\C T}^{n}(\Sigma,\Sigma)\]
compatible with the extended Yoneda product
for all $n$, also denoted by $\Sigma$.
It is easy to see that, using the notation of the previous section, 
\[
\modulesfp{\C T}_\Sigma=\modulesfp{\C T_\Sigma},\qquad 
(\widehat{\ext}_{\C T}^{n})_{\Sigma}\cong \widehat{\ext}_{\C T_\Sigma}^{n,*}.\]
The extension of the Yoneda product endows $\widehat{\ext}_{\C T_{\Sigma}}^{\bullet,*}$ with a bigraded monoid structure in the category of $\modulesfp{\C T_{\Sigma}}$-bimodules, hence 
\[\hh{\star}{\modulesfp{\C T},\widehat{\ext}_{\C T_{\Sigma}}^{\bullet,*}}\]
is a trigraded ring. Its units are concentrated in Hochschild degree $\star=0$ since $\star\geq 0$, while $\bullet$ and $*$ may be arbitrary integers.

\begin{proposition}\label{Heller_ungraded}
	If $\C T$ is a small ungraded idempotent complete additive category such that $\modulesfp{\C T}$ is Frobenius abelian and $\Sigma\colon \C T\rightarrow\C T$ is an automorphism, then the set of Puppe triangulated structures on $\C T$ with suspension functor $\Sigma$ is in bijection with the units of the bigraded ring $\hh{0}{\modulesfp{\C T},\widehat{\ext}_{\C T_{\Sigma}}^{\bullet,*}}$ lying in 
	\[\hh{0}{\modulesfp{\C T},\widehat{\ext}_{\C T_{\Sigma}}^{3,-1}}\cong \hh{0}{\modulesst{\C T},\homst_{\C T}(\Sigma,S^{3})}\] 
	and satisfying Heller's anticommutativity condition.
\end{proposition}

\begin{proof}
	Both the condition of a natural transformation $\Sigma\rightarrow S^3$ being invertible or the corresponding cocycle on the left being a unit can be translated in saying that the extension class $\varphi(M)$ associated to a finitely presented right $\C T$-module $M$ can be represented by an extension with representable (i.e.~projective-injective) middle terms
	\[\Sigma^{-1}M\hookrightarrow\C T(-,X)\rightarrow\C T(-,Y)\rightarrow\C T(-,Z)\twoheadrightarrow M.\]
	(In general, only either the two ones on the left or the two ones on the right can be taken to be representable.)
\end{proof}

This new glimpse at Heller's result allows to place the set of Puppe triangulated structures in a smaller recipient. First, note that Proposition \ref{graded_ungraded_long_exact_sequence} yields an exact sequence
\begin{equation}\label{exact_cohomology_sequence}
	\begin{tikzcd}[row sep=5mm]
		0\arrow[d, hook]\\
		\hh{0,-1}{\modulesfp{\C T_\Sigma},\widehat{\ext}_{\C T_\Sigma}^{3,*}}\arrow[d, "i^*"]\\
		\hh{0,-1}{\modulesfp{\C T},\widehat{\ext}_{\C T_\Sigma}^{3,*}}\arrow[d, "\operatorname{id}-\Sigma^{-1}_*\Sigma^*"]\\ 
		\hh{0,-1}{\modulesfp{\C T},\widehat{\ext}_{\C T_\Sigma}^{3,*}}
	\end{tikzcd}
\end{equation}
where the bottom map coincides with
\begin{equation*}
\begin{tikzcd}[row sep=5mm]
\hh{0}{\modulesfp{\C T},{\ext}_{\C T}^{3}(-,\Sigma^{-1})}\arrow[d, "\operatorname{id}+\Sigma^{-1}_*\Sigma^*"]\\ 
\hh{0}{\modulesfp{\C T},{\ext}_{\C T}^{3}(-,\Sigma^{-1})}
\end{tikzcd}
\end{equation*}
The apparent change of sign in the morphism is motivated by the definition of the graded bimodule morphism $\Sigma$ on the coefficients.

We now consider the trigraded ring
\[\hh{\star,*}{\modulesfp{\C T_\Sigma},\widehat{\ext}_{\C T_{\Sigma}}^{\bullet,*}},\]
which, as above, only contains units in the Hochschild degree $\star=0$ subring.

\begin{corollary}\label{Heller_graded}
	If $\C T$ is a small ungraded idempotent complete additive category such that $\modulesfp{\C T}$ is Frobenius abelian and $\Sigma\colon \C T\rightarrow\C T$ is an automorphism, then the	 set of Puppe triangulated structures on $\C T$ with suspension functor $\Sigma$ is in bijection with the units of the bigraded ring $\hh{0,*}{\modulesfp{\C T_\Sigma},\widehat{\ext}_{\C T_{\Sigma}}^{\bullet,*}}$ lying in 
	\[\hh{0,-1}{\modulesfp{\C T_\Sigma},\widehat{\ext}_{\C T_{\Sigma}}^{3,*}}\] whose image along $i^*$ satisfies Heller's anticommutativity condition.
\end{corollary}

\begin{proof}
	We are in the conditions of Remark \ref{multiplicative_properties}, hence $\hh{0,*}{\modulesfp{\C T_\Sigma},\widehat{\ext}_{\C T_{\Sigma}}^{\bullet,*}}$ is the equalizer of a pair of ring endomorphisms of $\hh{0,*}{\modulesfp{\C T},\widehat{\ext}_{\C T_{\Sigma}}^{\bullet,*}}$. This shows that the inclusion $i^*$ not only preserves but also reflects units.
	
	It remains to check that any Heller element $\delta\in \hh{0}{\modulesfp{\C T},\widehat{\ext}_{\C T_{\Sigma}}^{3,-1}}$ is in the kernel of $\operatorname{id}-\Sigma^{-1}_*\Sigma^*$. By the cohomological characterization of Heller's anticommutativity condition, see \eqref{anticommutativity_cohomological}, $\delta$ is in the kernel of $\operatorname{id}-\tau^{-1}_*S^*$. Using the natural isomorphism $\Sigma\cong S^3$ provided by $\delta$, we have a square of functors commuting up to natural isomorphism and two commutative squares of bimodules as follows,
	\begin{center}
		\begin{tikzcd}
			\modulesfp{\C T}\arrow[r,"\Sigma"]\arrow[d, two heads]&\modulesfp{\C T}\arrow[d, two heads]\\
			\modulesst{\C T}\arrow[r,"S^3"]&\modulesst{\C T}
		\end{tikzcd}
	\qquad
		\begin{tikzcd}
			{\ext}_{\C T}^{3}(-,\Sigma^{-1})\arrow[r,"-\Sigma"]\arrow[d, "\cong"']&{\ext}_{\C T}^{3}(\Sigma,-)\arrow[d, "\cong"]\\
			\widehat{\ext}_{\C T_\Sigma}^{3,-1}\arrow[r,"\Sigma"]\arrow[d, "\cong"']&\widehat{\ext}_{\C T_\Sigma}^{3,-1}(\Sigma,\Sigma)\arrow[d, "\cong"]\\
			\homst_{\C T}(\Sigma, S^3)\arrow[r,"\tau^3"]&\homst_{\C T}(\Sigma S^3, S^6)
		\end{tikzcd}
	\end{center}
	Hence we can identify $\operatorname{id}-\Sigma^{-1}_*\Sigma^*$ with $\operatorname{id}-\tau^{-3}_*(S^3)^*$. The result follows since the kernel of $\operatorname{id}-\tau^{-1}_*S^*$ is clearly contained in the kernel of $\operatorname{id}-\tau^{-3}_*(S^3)^*$.
\end{proof}

Since Heller's anticommutativity condition reduces to being in the kernel of a morphism, we can therefore think that the set of Puppe triangulated structures is a `locally closed' subset of the `affine space' $\hh{0,-1}{\modulesfp{\C T_\Sigma},\widehat{\ext}_{\C T_{\Sigma}}^{3,*}}$. (This is literal in the cases where the latter is a finite dimensional vector space over some field $k$, and this happens under appropriate finiteness assumptions on a $k$-linear $\C T$.) We will later give a neater cohomological characterization of the anticommutativity condition in terms of the cohomology of $\modulesfp{\C T_\Sigma}$ alone, see Sections \ref{spectral_sequence_section} and \ref{topological_section}.

\section{Toda brackets}\label{toda_brackets_section}

Together with his classification of Puppe triangulated structures in terms of natural transformations, recalled in the previous section, Heller embedded these triangulated structures into a set of operations called Toda brackets \cite[Theorem 13.2]{heller_stable_1968}. In this section we see that our new cohomological approach to Heller's theory fits perfectly with the Toda bracket perspective.

\begin{definition}\label{Toda_bracket}
	Let $\C{T}$ be an additive category and $\Sigma\colon\C{T}\rightarrow \C{T}$ an automorphism. A \emph{Toda bracket} $\langle -,-,-\rangle$ is an operation which sends three composable morphisms
	\[X\st{f}\To Y\st{g}\To Z\st{h}\To T\]
	with $g\cdot f=0$ and $h\cdot g=0$ to an element
	\[\langle h,g,f\rangle
	\in\frac{\C{T}( X,\Sigma^{-1}T)}{(\Sigma^{-1}h)\cdot\C{T}( X, \Sigma^{-1}Z)+\C{T}(Y,\Sigma^{-1}T)\cdot f},\]
	often regarded as a subset $\langle h,g,f\rangle\subset \C{T}(X,\Sigma^{-1}T)$.
	The following axioms must hold, whenever the Toda brackets are defined,
	\begin{align*}
		\langle i, h,g\rangle\cdot  f&\subset\langle i, h,g\cdot f\rangle,&		\langle i,h\cdot g, f\rangle&\supset\langle i\cdot h,g,f\rangle,\\
		\langle i, h,g\cdot f\rangle&\subset\langle i, h\cdot g, f\rangle,&
		\langle i\cdot h, g, f\rangle&\supset (\Sigma^{-1}i)\cdot \langle h,g,f\rangle.
	\end{align*}
	We say that a Toda bracket is \emph{stable} if in addition
	\begin{align*}
		\langle \Sigma h,\Sigma g,\Sigma f\rangle&=- \Sigma \langle h,g,f\rangle.
	\end{align*}
	The set of Toda brackets on $(\C T,\Sigma)$ is an abelian group with sum given by pointwise addition (even a $k$-module if our category is $k$-linear, since Toda brackets can be rescaled). It will be denoted by
	\[TB(\C{T},\Sigma).\]
	Stable Toda brackets form a subgroup (or submodule) denoted by
	\[TB^s(\C{T},\Sigma).\]
\end{definition}

Usually, the recipient of a Toda bracket is equivalently taken to be the isomorphic group
\[\frac{\C{T}(\Sigma X,T)}{h\cdot\C{T}(\Sigma X, Z)+\C{T}(\Sigma Y,T)\cdot(\Sigma f)}.\]
Our convention here, however, fits better with the rest of this paper. 

The well-known Toda bracket of a triangulated structure is stable, and it is determined by the previous laws and the fact that the Toda bracket of an exact triangle contains the identity. Conversely, this also defines the triangulated structure from the Toda bracket.

The last result of the previous section places the set of Puppe triangulated structures on $(\C T, \Sigma)$ within a graded Hochschild cohomology group, smaller than Heller's set of natural transformations, which has also been reinterpreted as a larger ungraded Hochschild cohomology group. In the following result we show that Heller's set of natural transformations is in bijection with the set of Toda brackets, and stable ones are in bijection with our smaller Hochschild cohomology group.

\begin{theorem}\label{Toda_brackets_and_cohomology}
	Given an idempotent complete additive category $\C T$ such that the category $\modulesfp{\C T}$ is abelian and an automorphism $\Sigma\colon\C T\rightarrow\C T$, there are isomorphisms compatible with the inclusions,
	\begin{align*}
		\hh{0}{\modulesfp{\C T},\ext_{\C T}^3(-,\Sigma^{-1})}&\cong TB(\C{T},\Sigma),\\
		\hh{0,-1}{\modulesfp{\C T_\Sigma},\ext_{\C T_{\Sigma}}^{3,*}}&\cong TB^s(\C{T},\Sigma).
	\end{align*}
\end{theorem}

\begin{proof}
	The second isomorphism follows from the first one together with Corollary \ref{graded_ungraded_long_exact_sequence} and the exact sequence \eqref{exact_cohomology_sequence} since under the first isomorphism, defined below, the stability condition for Toda brackets is equivalent to being in the kernel of $1-\Sigma_*^{-1}\Sigma^*$.

	We now start with the first cohomological interpretation of Toda brackets. A sequence of maps in $\C T$
	\[X\st{f}\To Y\st{g}\To Z\st{h}\To T\]
	with $g\cdot f=0$ and $h\cdot g=0$ is the same as a chain complex $P_*$ of projectives in $\modulesfp{\C T}$ concentrated in degrees from $0$ to $3$, and the Toda bracket is an element $\varphi(P_*)\in H^3(P_*,\Sigma^{-1}H_0(P_*))$. Here and below $\Sigma$ is not the classical suspension of chain complexes, but the exact invertible endofunctor $\Sigma$ of $\modulesfp{\C T}$. The four properties a Toda bracket amount to saying that $\varphi(P_*)$ is natural with respect to chain maps $Q_*\rightarrow P_*$ between such complexes. Hence, $\varphi$ is the same as a cocycle in
	\[\hh{0}{\chain_{\geq 0}^{\leq 3}(\C T),H^3(-,\Sigma^{-1}H_0(-))},\]
	where $\chain_{\geq 0}^{\leq 3}(\C T)$ is the category of chain complexes as above. We will now define a sequence of isomorphisms
	\begin{align*}
		\hh{0}{\chain_{\geq 0}^{\leq 3}(\C T),H^3(-,\Sigma^{-1}H_0(-))}&\cong \hh{0}{\chain_{\geq 0}(\C T),H^3( w_{\leq 3}(-),\Sigma^{-1}H_0(-))}\\
		&\cong \hh{0}{\chain_{\geq 0}(\C T),H^3(-,\Sigma^{-1}H_0(-))}\\
		&\cong \hh{0}{\derived_{\geq 0}(\C T),
		\derived_{\geq 0}(\C T)(-,\Sigma^{-1}H_0(-)[3])}\\
		&\cong \hh{0}{\modulesfp{\C T},\ext_{\C T}^3(-,\Sigma^{-1})}.
	\end{align*}
	Here $\chain_{\geq 0}(\C T)$ is the category of non-negative chain complexes of projectives in $\modulesfp{\C T}$. It fits in an adjoint pair
	\begin{center}
		\begin{tikzcd}[column sep=30mm]
			\chain_{\geq 0}^{\leq 3}(\C T)\arrow[r, shift left, "\text{inclusion}"]&\chain_{\geq 0}(\C T)\arrow[l, shift left, "\text{naive truncation }w_{\leq 3}"]
		\end{tikzcd}
	\end{center}
	which yields the first isomorphism, induced by $w_{\geq 3}$, compare \cite[Theorem 5.10]{muro_functoriality_2006}. 
	
	The $\chain_{\geq 0}(\C T)$-bimodule $H^3(w_{\leq 3}(-),\Sigma^{-1}H_0(-))$ is
	\begin{multline*}
	H^3(w_{\leq 3}(P_*),\Sigma^{-1}H_0(Q_*))\\=\coker[\hom_{\C T}(P_2,\Sigma^{-1}H_0(Q_*))\rightarrow\hom_{\C T}(P_3,\Sigma^{-1}H_0(Q_*))],
	\end{multline*}
	hence we have a short exact sequence of $\chain_{\geq 0}(\C T)$-bimodules
	\[H^3(-,\Sigma^{-1}H_0(-))\hookrightarrow H^3(w_{\leq 3}(-),\Sigma^{-1}H_0(-))\twoheadrightarrow M,\]
	where $M$ is defined by the images
	\[M(P_*,Q_*)=\im[\hom_{\C T}(P_3,\Sigma^{-1}H_0(Q_*))\rightarrow\hom_{\C T}(P_4,\Sigma^{-1}H_0(Q_*))].\]
	We now check that
	\[\hh{0}{\chain_{\geq 0}(\C T),M}=0,\]
	so the second isomorphism follows from the long exact cohomology sequence associated to the previous short exact sequence of coefficient bimodules. Indeed, given a cocycle $\psi\in \hh{0}{\chain_{\geq 0}(\C T),M}$ and an object $P_*$ in $\chain_{\geq 0}(\C T)$, we can consider the complex $Q_*$ which reduces to $P_4$ in degrees $3$ and $4$ (the only possibly non-trivial differential being the identity) and the only map $f\colon Q_*\rightarrow P_*$ which is the identity in degree $4$. The cocycle condition says that
	\[f\cdot \psi(Q_*)=\psi(P_*)\cdot f\in M(Q_*,P_*)=\hom_{\C T}(P_4,\Sigma^{-1}H_0(P_*)).\]
	Moreover, the map $M(P_*,P_*)\rightarrow M(Q_*,P_*)$ defined by right multiplication by $f$ is the inclusion of the image, hence injective, and $\psi(Q_*)=0$ since $H_0(Q_*)=0$, therefore $\psi(P_*)=0$, so $\psi=0$. 
	
	The dervied category $\derived_{\geq 0}(\C T)$ of non-negative chain complexes of finitely presented $\C T$-modules comes equipped with a canonical functor $\chain_{\geq 0}(\C T)\rightarrow \derived_{\geq 0}(\C T)$. The pull-back along this functor of the 
	bimodule
	$\derived_{\geq 0}(\C T)(-,\Sigma^{-1}H_0(-)[3])$, where $[3]$ denotes the $3$-fold shift in the derived category, is $H^3(-,\Sigma^{-1}H_0(-))$.  Moreover, the canonical functor factors through an equivalence from the (quotient) homotopy category of the source to the target. Hochschild cohomology computes ends, 
	and they can be equally computed in $\chain_{\geq 0}(\C T)$ or in the `quotient' $\derived_{\geq 0}(\C T)$. Hence we obtain the third isomorphism, induced by the previous canonical functor.
	
	We now consider the adjunction
	\begin{center}
	\begin{tikzcd}[column sep=30mm]
		\derived_{\geq 0}(\C T)\arrow[r, shift left, "H_0"]&\modulesfp{\C T}.\arrow[l, shift left, "\text{inclusion in degree }0"]
	\end{tikzcd}
	\end{center}
	It is actually a reflection. The left $\derived_{\geq 0}(\C T)$-module structure on the coefficient bimodule 
	$\derived_{\geq 0}(\C T)(-,\Sigma^{-1}H_0(-)[3])$ factors through $H_0$, hence one can check as in \cite[Theorem 5.4]{muro_functoriality_2006} that the degree $0$ inclusion induces an isomorphism in cohomology, the last one above.
	
	Unwrapping the previous isomorphisms, the Toda bracket associated to a cocycle $\varphi\in \hh{0}{\modulesfp{\C T},\ext_{\C T}^3(-,\Sigma^{-1})}$ can be computed as follows. Let  
	\[X\st{f}\To Y\st{g}\To Z\st{h}\To T\]
	be a sequence of maps in $\C T$
	with $g\cdot f=0$ and $h\cdot g=0$. We pick up a projective resolution $\C T(-,U_*)$ of $M=\coker\C T(-,h)$ in $\modulesfp{\C T}$. By standard homological algebra, there is a map of complexes, unique up to chain homotopy,
	\begin{center}
		\begin{tikzcd}[column sep=6.5mm]
			\cdots\arrow[r]&
			0\arrow[r]\arrow[d]&
			\C T(-,X)\arrow[r,"{\C T(-,f)}"]\arrow[d]&
			\C T(-,Y)\arrow[r, "{\C T(-,g)}"]\arrow[d]&
			\C T(-,Z)\arrow[r, "{\C T(-,h)}"]\arrow[d]&
			\C T(-,T)\arrow[d]\\
			\cdots\arrow[r]&
			\C T(-,U_4)\arrow[r]&
			\C T(-,U_3)\arrow[r]&
			\C T(-,U_2)\arrow[r]&
			\C T(-,U_1)\arrow[r]&
			\C T(-,U_0)&
		\end{tikzcd}
	\end{center}
	which induces the identity in $0$-dimensional homology (it is $M$ in both cases). This chain map induces a morphism in $3$-dimensional cohomology with coefficients in $\Sigma^{-1}M$,
	\[\ext_{\C T}^3(M,\Sigma^{-1}M)\To 
	\frac{\C{T}( X,\Sigma^{-1}T)}{(\Sigma^{-1}h)\cdot\C{T}( X, \Sigma^{-1}Z)+\C{T}(Y,\Sigma^{-1}T)\cdot f}.\]
	The Toda bracket $\langle f,g,h\rangle$ is the image of $\varphi(M)$ along this morphism.
\end{proof}

\section{A local-to-global spectral sequence}\label{spectral_sequence_section}

We now construct the spectral sequence which defines the first obstructions for the existence of an enhancement of a triangulated category over a field.

\begin{proposition}\label{spectral_sequence}
	If $\C C$ is a small graded category over a field $k$, there is a first quadrant cohomological spectral sequence of graded $k$-modules
	\[E_2^{p,q}=\hh{p,*}{\modulesfp{\C C},\ext_{\C C}^{p,*}}\Longrightarrow\hh{p+q,*}{\C C,\C C}.\]
\end{proposition}

\begin{proof}
	The spectral sequence will be associated to the bicomplex
	\[C^{\star,\bullet}=\hom_{\modulesfp{\C C}^{\env}}^*(B_{\star}(\modulesfp{\C C}),\hom_{\C C}^*(-\otimes_{\C C}B_{\bullet}(\C C),-)).\]
	An element of $C^{p,q}$ is the same a family of graded $k$-module morphisms
	\[\bigotimes_{i=1}^p\hom^*_{\C C}(M_{i},M_{j-1})
	\otimes M_p(X_0)\otimes
	\bigotimes_{j=1}^q\C C(X_{j},X_{j-1})\longrightarrow M_0(X_q)\]
	indexed by all sequences of objects $M_0,\dots, M_p$ in $\modulesfp{\C C}$ and $X_0,\dots, X_q$ in $\C C$, and, with this description, the horizontal and vertical differentials are
	\begin{equation}\label{differentials_bicomplex}
	\begin{split}
		d_h(\varphi)(g_1,\dots, g_{p+1},x,f_1,\dots, f_q)={}&(-1)^{\abs{\varphi}\abs{g_1}}g_1(X_q)(\varphi(g_2,\dots, g_{p+1},x,f_1,\dots, f_q))\\
		&+\sum_{i=1}^p(-1)^i\varphi(\dots,g_ig_{i+1},\dots,x,f_1,\dots, f_q)\\
		&+(-1)^{p+1}\varphi(g_1,\dots, g_{p+1}(X_0)(x),f_1,\dots, f_q),\\
		d_v(\varphi)(g_1,\dots, g_{p},x,f_1,\dots, f_{q+1})={}&\varphi(g_1,\dots, g_{p},x\cdot f_1,\dots, f_{q+1})\\
		&+\sum_{i=1}^q(-1)^i\varphi(g_1,\dots, g_{p},x,\dots, f_if_{i+1},\dots)\\
		&+(-1)^{q+1}\varphi(g_1,\dots, g_{p},x,f_1,\dots, f_q)\cdot f_{q+1}.
	\end{split}
	\end{equation}
	 
	Since we are working over a field, all categories are \emph{locally free}, i.e.~morphism graded modules are free, and bar resolutions are honest projective bimodule resolutions. Moreover, for any right $\C C$-module $M$, the complex of right $\C C$-modules $M\otimes_{\C C}B_{\bullet}(\C C)$ is a projective resolution. Here we use that augmented bar resolutions, in general, admit a contraction as complexes of left or right modules (not as complexes of bimodules), so the homology of $M\otimes_{\C C}B_{\bullet}(\C C)$ is $M$ is concentrated in degree $0$. We also use that 
	\[M\otimes_{\C C}B_{n}(\C C)
	=
	\bigoplus_{X_0,\dots, X_n} M(X_0)\otimes\C C(X_1,X_{0})\otimes\cdots\otimes\C C(X_n,X_{n-1})\otimes\C C(-,X_n)\]
	and $M$ takes free values since the ground ring is a field, so this is a direct sum of representables, hence projective.
	
	Since $\modulesfp{\C C}$ is locally free $B_{\star}(\modulesfp{\C C})$ is a \emph{projective} resolution of $\modulesfp{\C C}$ as a bimodule over itself, so the $E_2$-term of the first-vertical-then-horizontal cohomology spectral sequence is
	\[E_2^{p,q}=\hh{p,*}{\modulesfp{\C C},\ext_{\C C}^{p,*}}.\]
	In order to compute the target of this spectral sequence, i.e.~the cohomology of the total complex of $C^{\star,\bullet}$, we now look at the $E_2$-term of the other spectral sequence associated to the bicomplex.
	
	The $\modulesfp{\C C}$-bimodule $\hom_{\C C}^*(-\otimes_{\C C}B_{n}(\C C),-)$ sends $M$ and $N$ to
	\[\prod_{X_1,\dots,X_n}\hom^*_k(M(X_0)\otimes\C C(X_1,X_{0})\otimes\cdots\otimes\C C(X_n,X_{n-1}),N(X_n)),\]
	hence it is a product of $\modulesfp{\C C}$-bimoules $D_{X,Y}$ of the form \[D_{X,Y}(M,N)=\hom_{k}^*(M(X),N(Y)),\] where $X$ and $Y$ are fixed objects in $\C C$. Here we use again that we are working over a field, so $k$-module morphism objects in $\C C$ are free. The cohomology of $\modulesfp{\C C}$ with coefficients in such a $D_{X,Y}$ vanishes in positive dimensions. For this, we use the chain homotopy $h$ defined on 
	$\hom_{\C C^{\env}}^*(B_\bullet(\modulesfp{\C C}),D_{X,Y})$ as follows, compare \cite[Lemma 3.10]{jibladze_cohomology_1991}. Given an $(n+1)$-cochain $\varphi$ and morphisms $f_i\colon M_i\rightarrow M_{i-1}$ in $\modulesfp{\C C}$, $n\geq 0$, $1\leq i\leq n$, $h(\varphi)(f_1\otimes \cdots\otimes f_n)\in D_{X,Y}(M_n,M_0)$ is defined by
	\begin{align*}
	h(\varphi)(f_1\otimes \cdots\otimes f_n)\colon \hom_{\C C}^*(\C C(-,X),M_n)&=M_n(X)\To M_0(Y),\\
	g&\longmapsto \varphi(f_1\otimes \cdots\otimes f_n\otimes g)(1_X).
	\end{align*}
	Hence, the inclusion of the $0$-dimensional horizontal cohomology in $C^{\star,\bullet}$ is a quasi-isomorphism (we mean with the total complex of $C^{\star,\bullet}$). This $0$-dimensional horizontal cohomology is the end of the cochain complex of $\modulesfp{\C C}$-bimodules $\hom_{\C C}^*(-\otimes_{\C C}B_{\bullet}(\C C),-)$. Such end is, dimension-wise, the graded module of graded natural transformations from the source to the target regarded as graded functors $\modulesfp{\C C}\rightarrow \modules{\C C}$. The source preserves colimits, and $\C C\subset\modulesfp{\C C}$ is the inclusion of a dense subcategory \cite[\S5.1]{kelly_basic_2005}, hence the source is the left Kan extension of its restriction along $\C C\subset\modulesfp{\C C}$ \cite[Theorem 5.29]{kelly_basic_2005}, so the end can be computed by restricting to $\C C$. The latter end is the complex $\hom_{\C C^{\env}}^*(B_\bullet(\C C),\C C)$, whose cohomology is the claimed target of the spectral sequence. 
	
	The explicit quasi-isomorphism $\xi\colon \hom_{\C C^{\env}}^*(B_\bullet(\C C),\C C)\hookrightarrow C^{\star,\bullet}$ is defined as follows. Given an $n$-cochain $\varphi$ in the source, a finitely presented $\C C$-bimodule $M$, morphisms $f_i\colon X_i\rightarrow X_{i-1}$ in $\C C$, $n\geq 0$, $1\leq i\leq n$, and $g\in M(X_0)$, 
	\begin{equation}\label{quasi_iso}
		\xi(\varphi)(M)(g, f_1,\cdots, f_n)=(-1)^{\abs{\varphi}\abs{g}}g\varphi(f_1,\cdots, f_n).
	\end{equation}
\end{proof}

\begin{remark}
	In the previous proposition we can replace $\modulesfp{\C C}$ with any small full subcategory $\C B\subset\modules{\C C}$ containing $\C C$ since then $\C C$ is dense in $\C B$. 
	
	The spectral sequence has an ungraded version, where $\C C$ is an ungraded category and $\modulesfp{\C C}$ is replaced with the category of finitely presented ungraded right $\C C$-modules, or any full subcategory $\C B$ of the category of ungraded $\C C$-modules containing $\C C$. The proof is exactly the same. 
	
	The ungraded spectral sequence has been considered in \cite[Theorem 5.4.1]{lowen_hochschild_2005}. 	There, $\C C$ is replaced with the category of injective objects in an Grothendieck abelian category, but this is essentially equivalent to our framework because the Hochschild cohomology of a category coincides with that of its opposite.
\end{remark}

We now prove that, for a particular instance of the previous spectral sequence, the edge morphism takes the universal Massey product of an $A$-infinity model of a triangulated category to the Toda bracket which classifies the triangulated structure.

\begin{theorem}\label{theorem_edge_morphism}
	Let $\C T$ be an idempotent complete triangulated category over a field $k$ with suspension $\Sigma$. Assume we have a minimal $A$-infinity enhancement with universal Massey product $\{m_3\}\in \hh{3,-1}{\C T_{\Sigma}, \C T_{\Sigma}}$. Then the edge morphism 
	\[\hh{3,-1}{\C T_{\Sigma}, \C T_{\Sigma}}\To \hh{0,-1}{\modulesfp{\C T_{\Sigma}},\ext_{\C T_{\Sigma}}^{3,\ast}}
	\]
	of the spectral sequence in Proposition \ref{spectral_sequence} for $\C C=\C T_{\Sigma}$ takes the universal Massey product to the Toda bracket of the triangulated structure.
\end{theorem}

\begin{proof}
	Given an exact triangle
	\[X\stackrel{f}{\To}Y\stackrel{i}{\To}C_f\stackrel{q}{\To}\Sigma X\]
	in $\C T$, 
	\begin{center}
	\begin{tikzcd}[column sep=5mm]
		\cdots\arrow[r]&
		\C T(-,\Sigma^{-1}C_f)\arrow[r, "{\C T(-,\Sigma^{-1}q)}" yshift=4pt]\arrow[r] &
		\C T(-,X)\arrow[r,"{\C T(-,f)}" yshift=4pt] &
		\C T(-,Y)\arrow[r, "{\C T(-,i)}" yshift=4pt] &
		\C T(-,C_f)\arrow[r, "{\C T(-,q)}" yshift=4pt] &
		\C T(-,\Sigma X) 
	\end{tikzcd}
\end{center}
	is a projective resolution of $\coker\C T(-,q)$ in $\modulesfp{\C T}$. 
	Combining this fact with the explicit formulas in the proofs of Theorem \ref{Toda_brackets_and_cohomology} and Proposition \ref{spectral_sequence}, we see that the Toda bracket defined by the image of $\{m_3\}$ along the edge morphism in the statement and the isomorphism in Theorem \ref{Toda_brackets_and_cohomology} satisfies
	\[m_3(h,g,f)\in\langle h,g,f\rangle.\]
	By \cite{bondal_enhanced_1991}, this is the Toda bracket associated to the triangulated structure of $\C T$ with suspension $\Sigma$, since it is enhanced by the minimal $A$-infinity category structure with universal Massey product $\{m_3\}$.
\end{proof}

Next, we define certain morphisms within the $E_2$-term of the previous spectral sequence. These morphisms will be used afterwards to give a new cohomological interpretation of Heller's anticommutativity condition.

\begin{definition}\label{kappa}
	Let $\C C$ be a graded category such that $\modulesfp{\C C}$ is graded Frobenius abelian. We define the graded $k$-module morphisms, $p\geq 1$, $q\in\mathbb Z$,
	\[\kappa \colon \hh{p+1,*}{\modulesfp{\C C },\widehat{\ext}_{\C C}^{q,*}}
	\To
	\hh{0,*}{\modulesfp{\C C },\widehat{\ext}_{\C C}^{p+q,*}}\]
	as follows. Given a cohomology class $\{\psi\}$ in the source represented by the cocycle $\psi$ and a finitely presented right $\C C$-module $M$, we consider the $p$-extension in $\modulesfp{\C C}$
	\[M\stackrel{f_0}\hookrightarrow P_0\stackrel{f_1}\longrightarrow \cdots\stackrel{f_{p-1}}\longrightarrow P_{p-1}\stackrel{f_p}\twoheadrightarrow S^{p}M\]
	with middle projective-injective terms obtained as the Yoneda product of the extensions defining $S^iM$, $1\leq i\leq p$, we evaluate the cocycle $\psi$ on it
	\[\psi(f_p,\cdots, f_0)\in \widehat{\ext}_{\C C}^{q,*}(M,S^pM)\cong 
	\widehat{\ext}_{\C C}^{p+q,*}(M,M)\]
	and define $\kappa(\{\psi\})(M)$ as the corresponding element on the right of the isomorphism. This isomorphism is also induced by the short exact sequences with projective-middle terms defining $S^iM$, $1\leq i\leq p$. It is easy to check that $\kappa(\{\psi\})$ does not depend on the choice of representative $\psi$, since $f_{i+1}f_i=0$ and all the $P_i$ are projective-injective. This must also be used to check that $\kappa(\{\psi\})$ is indeed a $0$-cocycle.
\end{definition}

\begin{proposition}\label{Heller_criterion_by_differential}
	If $\C T$ is a small ungraded additive category over a field $k$ such that $\modulesfp{\C T}$ is Frobenius abelian and $\Sigma\colon \C T\rightarrow\C T$ is an automorphism, then the	 set of Puppe triangulated structures on $\C T$ with suspension functor $\Sigma$ is in bijection with the units of the bigraded ring $\hh{0,*}{\modulesfp{\C T_\Sigma},\widehat{\ext}_{\C T_{\Sigma}}^{\bullet,*}}$ lying in the kernel of the composite
	\begin{center}
		\begin{tikzcd}
			\hh{0,-1}{\modulesfp{\C T_\Sigma},\ext_{\C T_{\Sigma}}^{3,*}}
			\arrow[d,"d_2"]\\
			\hh{2,-1}{\modulesfp{\C T_\Sigma},\ext_{\C T_{\Sigma}}^{2,*}}
			\arrow[d,"\kappa"]\\
			\hh{0,-1}{\modulesfp{\C T_\Sigma},\ext_{\C T_{\Sigma}}^{3,*}}
		\end{tikzcd}
	\end{center}
	where $d_2$ is a second differential in the spectral sequence of Proposition \ref{spectral_sequence}.
\end{proposition}

\begin{proof}
	Using \eqref{anticommutativity_cohomological}, Proposition \ref{Heller_ungraded}, \eqref{exact_cohomology_sequence}, and Corollary \ref{Heller_graded} and its proof, we see that it suffices to prove that the following diagram commutes up to sign,
	\begin{center}
		\begin{tikzcd}[column sep=50]
			\hh{0,-1}{\modulesfp{\C T_\Sigma},\ext_{\C T_{\Sigma}}^{3,*}}
			\arrow[r,"\kappa d_2"] \arrow[d, "i^*", hook]
			&\hh{0,-1}{\modulesfp{\C T_\Sigma},\ext_{\C T_{\Sigma}}^{3,*}}\arrow[d, "i^*", hook]
			\\
			\hh{0,-1}{\modulesfp{\C T},\ext_{\C T_{\Sigma}}^{3,*}}\arrow[d, "\cong"]
									&\hh{0,-1}{\modulesfp{\C T},\ext_{\C T_{\Sigma}}^{3,*}}\arrow[d, "\cong"]\\
			\hh{0}{\modulesst{\C T},\homst_{\C T}(\Sigma,S^{3})}\arrow[r, "\operatorname{id}-\tau_*^{-1}S^*"]&\hh{0}{\modulesst{\C T},\homst_{\C T}(\Sigma,S^{3})}
		\end{tikzcd}
	\end{center}
	
	Let $\varphi\in \hh{0,-1}{\modulesfp{\C T_\Sigma},\widehat{\ext}_{\C T_{\Sigma}}^{3,*}}$. Our spectral sequence is the spectral sequence of a bicomplex, hence it is straightforward, although a little bit tedious, to compute $d_2(\varphi)$. First of all, given a finitely presented right $\C T_\Sigma$-module $M$ we must represent $\varphi(M)\in \ext_{\C T_{\Sigma}}^{3,-1}(M,M)$ by a right $\C T_\Sigma$-module morphism of degree $-1$ 
	\[\tilde\varphi(M)\colon M\otimes_{\C T_{\Sigma}}B_3(\C T_{\Sigma})\To M,\]
	which, by the Yoneda lemma, is the same as a collection of degree $-1$ morphisms of graded $k$-modules
	\[\tilde\varphi(M)\colon M(X_0)\otimes\C T_{\Sigma}(X_1,X_0)
	\otimes\C T_{\Sigma}(X_2,X_1)
	\otimes\C T_{\Sigma}(X_3,X_2)
	\To
	M(X_3)
	\]
	indexed by objects $X_0,\dots, X_3$ in $\C T$. We can suppose without loss of generality that $\tilde\varphi(M)=0$ for $M$ projective-injective. 
	
	If we have a projective-injective resolution of $M$
	\begin{center}
	\begin{tikzcd}[column sep=5mm]
		\cdots\arrow[r]&
		\C T_{\Sigma}(-,X_3)\arrow[r,"d_3"]&
		\C T_{\Sigma}(-,X_2)\arrow[r,"d_2"]&
		\C T_{\Sigma}(-,X_1)\arrow[r,"d_1"]&
		\C T_{\Sigma}(-,X_0)\ar[r, two heads,"d_0"]&
		M,
	\end{tikzcd}
	\end{center}
    there is a morphism of resolutions $\C T_{\Sigma}(-,X_*)\rightarrow M\otimes_{\C T_{\Sigma}}B_*(\C T_\Sigma)$ defined degreewise by the elements 
	\[
    (-1)^{\frac{n(n+1)}{2}}d_0\otimes\cdots \otimes d_n\otimes\operatorname{id}_{X_n}
    \]
    in
    \[
    M(X_0)\otimes\C T_{\Sigma}(X_1,X_0)\otimes\cdots\otimes\C T_{\Sigma}(X_n,X_{n-1})\otimes\C T_{\Sigma}(X_n,X_n).
    \]
    Hence $\varphi(M)$ is also represented by the 
    morphism $\C T_{\Sigma}(-,X_3)\rightarrow M$ defined by 
    $\tilde{\varphi}(M)(d_0\otimes d_1\otimes d_2\otimes d_3)$.
	
	Since $\varphi$ is a Hochschild cocycle, for each $g\colon M\rightarrow N$ in $\modulesfp{\C T_{\Sigma}}$ we must have
	\[\psi(g)\colon M\otimes_{\C T_{\Sigma}}B_2(\C T_{\Sigma})\To N,\]
	of degree $-1$, i.e.
	\[\psi(g)\colon M(X_0)\otimes\C T_{\Sigma}(X_1,X_0)
	\otimes\C T_{\Sigma}(X_2,X_1)
	\To
	N(X_2)
	\]
	as above such that
	\begin{multline}\label{equation_phi_psi}
	(-1)^{\abs{g}}g\cdot \tilde{\varphi}(M)(x\otimes f_1\otimes f_2\otimes f_3)
	-\tilde{\varphi}(N)(g\cdot x\otimes f_1\otimes f_2\otimes f_3)={}
	\\
	\psi(g)(x\cdot f_1\otimes f_2\otimes f_3)
	-\psi(g)(x\otimes f_1\cdot f_2\otimes f_3)\\
	+\psi(g)(x\otimes f_1\otimes f_2\cdot f_3)
	-\psi(g)(x\otimes f_1\otimes f_2)\cdot f_3.
	\end{multline}
	Moreover, $\psi(g)$ must be $k$-linear in $g$.
	
	The cohomology class $d_2(\varphi)$ is represented by the $2$-cocycle $\xi$ such that,	
	given two composable morphisms in $\modulesfp{\C T}$
	\[M_2\stackrel{g_2}{\To}M_1\stackrel{g_1}{\To}M_0,\]
	the element $\xi(g_1\otimes g_2)\in \ext_{\C T_{\Sigma}}^{2,\abs{g_1}+\abs{g_2}-1}(M_2,M_0)$ is represented by the map
	\[(-1)^{\abs{g_1}}g_1\psi(g_2)-\psi(g_1\cdot g_2)+\psi(g_1)(g_2\otimes\operatorname{id})\colon  M_2\otimes_{\C T_{\Sigma}}B_2(\C T_{\Sigma})\To M_0.\]
	
	According to the definition of $\kappa$, we can decompose it in two steps, $\kappa=\kappa_2\kappa_1$, 
	\begin{align*}
	\hh{2,-1}{\modulesfp{\C T_{\Sigma} },\widehat{\ext}_{\C T_{\Sigma}}^{2,*}}
		&\stackrel{\kappa_1}\To
	\hh{0,-1}{\modulesfp{\C T_{\Sigma} },\widehat{\ext}_{\C T_{\Sigma}}^{2,*}(-,S)}\\
		&\mathop{\To}\limits_{\kappa_2}^{\cong}
	\hh{0,-1}{\modulesfp{\C T_{\Sigma} },\widehat{\ext}_{\C T_{\Sigma}}^{3,*}}.
	\end{align*}
	The second one, $\kappa_2$, is an isomorphism. 
	If we consider the short exact sequences defining $S$, 
	\[M\stackrel{j_M}\hookrightarrow\C T_{\Sigma}(-,X_M)\stackrel{p_M}\twoheadrightarrow S M,\]
	consisting of degree $0$ morphisms, 
	then $\kappa_1d_2(\varphi)(M)\in \widehat{\ext}_{\C T_{\Sigma}}^{2,-1}(M,SM)$ is represented by
	\[p_M\psi(j_M)+\psi(p_M)(j_M\otimes\operatorname{id})\colon  M\otimes_{\C T_{\Sigma}}B_2(\C T_{\Sigma})\To SM.\]
	With the small resolution $\C T_{\Sigma}(-,X_*)$ of $M$, $\kappa_1d_2(\varphi)(M)$ is represented by
	\begin{equation}\label{k_d2_phi}
	-p_M \psi(j_M)(d_0\otimes d_1\otimes d_2)-\psi(p_M) (j_M\cdot d_0\otimes d_1\otimes d_2)\in SM(X_2).
	\end{equation}
	Using \eqref{equation_phi_psi}, it is straightfoward to check that both summands represent elements in $\widehat{\ext}_{\C T_{\Sigma}}^{2,-1}(M,SM)$, i.e.~each of them vanishes when multiplying by $d_3$ on the right.
	
	The natural isomorphism 
	\begin{equation}\label{iso_S}
		\widehat{\ext}_{\C T_{\Sigma}}^{2,-1}(M,SN)\cong 
		\widehat{\ext}_{\C T_{\Sigma}}^{3,-1}(M,N)
	\end{equation}
	defining $\kappa_2$ is given as follows. We pick a representative $\zeta\colon\C T_\Sigma(-,X_2)\rightarrow SN$ of an element in the source, we lift $\zeta$ along $p_N$, $\zeta=p_N\zeta'$, $\zeta'd_3$ factors (uniquely) through the inclusion $j_N$, $\zeta'd_3=j_N\zeta''$, and $\zeta''\colon\C T_{\Sigma}(-,X_3)\rightarrow N$ represents the image of the isomorphism \eqref{iso_S}. Note that, if we plug the extension defining $SM$ to the previous resolution of $M$, we obtain a projective-injective resolution of $SM$,
	\begin{center}
	\begin{tikzcd}[column sep=5mm]
		\cdots\arrow[r]&
		\C T_{\Sigma}(-,X_2)\arrow[r,"d_2"]&
		\C T_{\Sigma}(-,X_1)\arrow[r,"d_1"]&
		\C T_{\Sigma}(-,X_0)\arrow[r,"j_Md_0"]&
		\C T_{\Sigma}(-,X_M)
		\ar[r, two heads,"p_M"]&
		SM,
	\end{tikzcd}
	\end{center}
	yielding an immediate identification $\widehat{\ext}_{\C T_{\Sigma}}^{2,-1}(M,SN)=\widehat{\ext}_{\C T_{\Sigma}}^{3,-1}(SM,SN)$. Moreover, using this identification, the isomorphism \eqref{iso_S} is $S^{-1}$. Hence, if we apply it to $\varphi(SM)$ we obtain $S^{-1}\varphi(SM)=-\tau^{-1}_*S^*(\varphi)(M)$.
	
	The first summand in \eqref{k_d2_phi} is already lifted along $p_M$. Moreover, by \eqref{equation_phi_psi},
	\begin{equation*}
		\psi(j_M)(d_0\otimes d_1\otimes d_2)\cdot d_3=j_M\cdot\tilde{\varphi}(M)(d_0\otimes d_1\otimes d_2\otimes d_3),
	\end{equation*}	
	and $\tilde{\varphi}(M)(d_0\otimes d_1\otimes d_2\otimes d_3)$ represents $\varphi(M)$. 
	Furthermore, again by \eqref{equation_phi_psi},
	\begin{equation*}
	\begin{split}
		\psi(p_M) (j_M\cdot d_0\otimes d_1\otimes d_2)=&-\tilde{\varphi}(SM)(p_M\otimes j_M\cdot d_0\otimes d_1\otimes d_2)\\&
		+\psi(p_M) (\operatorname{id}_{X_M}\otimes j_M\cdot d_0\otimes d_1)\cdot d_2,
	\end{split}
	\end{equation*}
	and $\tilde{\varphi}(SM)(p_M\otimes j_M\cdot d_0\otimes d_1\otimes d_2)$ represents $\varphi(SM)$. 
	Hence, by the previous paragraph, applying $\kappa_2$ to the element represented by the second summand in \eqref{k_d2_phi} we obtain $\tau^{-1}_*S^*(\varphi)(M)$. 
	Summing up, we have checked that $\kappa d_2(\varphi)(M)=-\varphi(M)+\tau^{-1}_*S^*(\varphi)(M)$. 
	This finishes the proof.
\end{proof}

\section{The octahedral axiom}\label{octahedral_axiom_section}

Verdier's octahedral axiom, unlike the rest of axioms for a triangulated category, does not seem to have an algebraic characterization. Nevertheless we have the following algebraic sufficient condition. After Proposition \ref{Heller_criterion_by_differential}, this sufficient condition can be regarded as a strengthening of Heller's anticommutativity condition. A rather close strengthening indeed. It actually reflects in a very precise way the known fact that the octahedral axiom is at the bottom of the coherence hierarchy of enhancements. The proof is rather lengthy because we must prove the octahedral axiom in a non-standard way, but we obtain as a corollary an interesting characterization of pre-triangulated DG- or $A$-infinity categories over a field, in the sense of Bondal and Kapranov.

\begin{theorem}\label{octahedral_theorem}
	If $\C T$ is a small ungraded idempotent complete additive category over a field $k$ such that $\modulesfp{\C T}$ is Frobenius abelian, $\Sigma\colon \C T\rightarrow\C T$ is an automorphism, and $\varphi\in \hh{0,-1}{\modulesfp{\C T_\Sigma},\ext_{\C T_{\Sigma}}^{3,*}}$ corresponds to a Puppe triangulated structure on $\C T$ with suspension $\Sigma$ such that $d_2(\varphi)=0$, then this Puppe triangulated structure satisfies the octahedral axiom.
\end{theorem}

\begin{proof}
	Neeman's mapping cone criterion \cite[Definition 1.3.13, Proposition 1.4.6, and Remark 1.4.7]{neeman_triangulated_2001} seems to be the preferred route to prove the octahedral axiom for triangulated categories with no (known) models \cite{amiot_structure_2007, muro_triangulated_2007}. 
		We must show that any map between the bases of two exact triangles can be extended to a triangle morphisms whose mapping cone is exact, i.e.~for any commutative diagram of solid arrows between exact triangles
	\begin{center}
		\begin{tikzcd}
			X\arrow[r,"f"]\arrow[d," h_1"]&Y\arrow[r,"i"]\arrow[d,"h_2"]&		C_f\arrow[r,"q"]\arrow[d,"h_3",dashed]&		\Sigma X\arrow[d,"\Sigma h_1"]\\
			X'\arrow[r,"f'"]&		
			Y'\arrow[r,"i'"]&
			C_{f'}\arrow[r,"q'"]	
			&\Sigma X'
		\end{tikzcd}
	\end{center}
	we can find a filler $h_3$ whose mapping cone 
	\begin{center}
		\begin{tikzcd}[column sep=16mm, ampersand replacement=\&]
			Y\oplus X'\arrow[r, "
			{\left(
				\begin{smallmatrix}
					-i&0\\
					h_2&f'
				\end{smallmatrix}
				\right)}
			"]\&C_f\oplus Y'\arrow[r,"{\left(\begin{smallmatrix}
					-q&0\\h_3&i'
				\end{smallmatrix}\right)}"]\& \Sigma X\oplus C_{f'}
			\arrow[r,"{\left(\begin{smallmatrix}
					-\Sigma f&0\\h_3&q'
				\end{smallmatrix}\right)}"]\& \Sigma Y\oplus \Sigma X'
		\end{tikzcd}
	\end{center}
	is an exact triangle. 
	
	We claim that, if this holds for $(h_1,h_2)$ then it also holds for $(h_1',h_2')$ provided there are maps
	\[X\stackrel{k_1}\longrightarrow T\stackrel{k_2}{\longrightarrow} X',\qquad \Theta\colon  Y\longrightarrow X',\]
	such that
	\begin{align*}
	f'k_2&=0,&
	h_1'-h_1-k_2k_1&=\Theta f.
	\end{align*}
	We check this claim by playing with Puppe's axioms. Indeed, by the first equation $k_2$ must factor through $\Sigma^{-1}q'$, so we can rephrase our conditions as follows: there exist maps 
	\[\Sigma^{-1}\Psi\colon X\longrightarrow \Sigma^{-1}C_{f'},\qquad \Theta\colon  Y\longrightarrow X',\]
	such that
	\begin{align*}
	h_1'-h_1&=\Theta f+\Sigma^{-1}(q'\Psi).
	\end{align*}
	We have
	\begin{equation*}
	\begin{split}
	(h_2'-h_2-f'\Theta)f&=(h_2'-h_2)f-f'\Theta f\\
	&=f'(h_1'-h_1)-f'(\Theta f)\\
	&=f'(h_1'-h_1-\Theta f)\\
	&=\underbrace{f'(\Sigma^{-1}q')}_{=0}(\Sigma^{-1}\Psi)=0.
	\end{split}
	\end{equation*}
	Therefore, there exists $\Phi\colon C_f\rightarrow Y'$ such that \[h_2'-h_2=\Phi i+f'\Theta.\] 
	We claim that, if we define \[h_3'=h_3+i'\Phi+\Psi q\] then $(h_1',h_2',h_3')$ is a morphism of triangles. We only have to check that the two squares on the right, those containing $h_3'$, commute. This amounts to
	\begin{align*}
	h_3'i&=(h_3+i'\Phi+\Psi q)i&
	q'h_3'&=q'(h_3+i'\Phi+\Psi q)
	\\
	&=h_3i+i'\Phi i+\Psi qi&
	&=q'h_3+q'i'\Phi+q'\Psi q\\
	&=i'h_2+i'\Phi i&
	&=(\Sigma h_1)q+q'\Psi q\\
	&=i'h_2+i'\Phi i+i'f'\Theta\quad {\scriptstyle (i'f'=0)}&
	&=(\Sigma h_1)q+q'\Psi q+\Sigma(\Theta f)q\quad {\scriptstyle ((\Sigma f)q=0)}\\
	&=i'(h_2+\Phi i+f'\Theta)&
	&=(\Sigma h_1+q'\Psi +\Sigma(\Theta f))q\\
	&=i'h_2',&
	&=(\Sigma h_1')q.\\
	\end{align*}
	The map of triangles $(h_1',h_2',h_3')$ is homotopic to $(h_1,h_2,h_3)$ in the sense of \cite[Definition 1.3.2]{neeman_triangulated_2001} by construction. This implies that the mapping cone of the former is isomorphic to the mapping cone of the later \cite[Lemma 1.3.3]{neeman_triangulated_2001}, hence it is also exact.
	
	
	Exact triangles can be made into a category whose maps are just morphisms between the bases, i.e.~pairs $(h_1,h_2)$ as above. We can even form the quotient category under the equivalence relation $(h_1,h_2)\sim(h_1',h_2')$ defined by the existence of $\Sigma^{-1}\Psi$ and $\Theta$ as above. The quotient category is equivalent to $\modulesst{\C T}$, and the equivalence is realized by the functor sending an exact triangle $(f,i,q)$ to $\ker \C T(-,f)$ (this kernel is taken in $\modulesfp{\C T}$). This can be easily checked by using the alternative descriptions of the abelianization $\modulesfp{\C T}$ in \cite[5.1]{neeman_triangulated_2001} and the construction of the quotient category $\modulesst{\C T}$. Hence, we have just proved that the property we have to check only depends on the morphism $\ker \C T(-,f)\rightarrow \ker \C T(-,f')$ in $\modulesst{\C T}$ induced by $(h_1,h_2)$. It does not even depend on the triangles $(f,i,q)$ and $(f',i',q')$ lifting $\ker \C T(-,f)$ and $\ker \C T(-,f')$ through the previous equivalence of categories.
	
	Let us denote $A=\ker \C T(-,f)$ and $A'=\ker \C T(-,f')$. Using the suspension functor $S$ in the stable module category we get 
	\[\homst_{\C T}(A,A')\cong \ext_{\C T}^{1}(S A,A').\] We can assume that $SA$ is defined by the short exact sequence
	\[A\hookrightarrow\C T(-,X)\twoheadrightarrow SA\]
	in $\modulesfp{\C T}$ arising from the exact triangle $(f,i,q)$, hence $SA$ is also $\ker \C T(-,i)$. The short exact sequence 
	\[A'\hookrightarrow B\twoheadrightarrow SA\]
	representing the extension corresponding to the map $A\rightarrow A'$ induced by $(h_1,h_2)$ can be obtained from the following diagram
	\begin{center}
		\begin{tikzcd}[column sep=16mm, ampersand replacement=\&]
			X'\arrow[r,"f'"]\arrow[d, hook]\&
			Y'	\arrow[d, hook]
				\\
			Y\oplus X'\arrow[d, two heads]\arrow[r, "
			{\left(
				\begin{smallmatrix}
					-i&0\\
					h_2&f'
				\end{smallmatrix}
				\right)}
			"]\&C_f\oplus Y'\arrow[d, two heads]\\
			Y\arrow[r,"-i"] \&		C_f
		\end{tikzcd}
	\end{center}
	first embedding it in $\modulesfp{\C T}$ through the Yoneda inclusion and then taking kernels of horizontal arrows. Here, the vertical arrows are the obvious inclusions and projections of factors of a direct sum. Indeed, by the following paragraph we obtain a short exact sequence when taking kernels.
	
	For any choice of filler $h_3$ in the very first diagram of this proof, the mapping cone fits in the following commutative diagram of triangles
	\begin{center}
		\begin{tikzcd}[column sep=16mm, ampersand replacement=\&]
			X'\arrow[r,"f'"]\arrow[d, hook]\&
			Y'\arrow[r,"f'"]	\arrow[d, hook]
			\&C_{f'}\arrow[d, hook]\arrow[r,"q'"]\&\Sigma	X'\arrow[d, hook]
			\\
			Y\oplus X'\arrow[d, two heads]\arrow[r, "
			{\left(
				\begin{smallmatrix}
					-i&0\\
					h_2&f'
				\end{smallmatrix}
				\right)}
			"]\&C_f\oplus Y'\arrow[d, two heads]\arrow[r,"{\left(\begin{smallmatrix}
					-q&0\\h_3&i'
				\end{smallmatrix}\right)}"]\& \Sigma X\oplus C_{f'}\arrow[d, two heads]
			\arrow[r,"{\left(\begin{smallmatrix}
					-\Sigma f&0\\\Sigma h_1&q'
				\end{smallmatrix}\right)}"]\& \Sigma Y\oplus \Sigma X'\arrow[d, two heads]\\
			Y\arrow[r,"-i"] \&		C_f\arrow[r,"-q"] \&		\Sigma X \arrow[r,"-\Sigma f"]\&\Sigma Y 
		\end{tikzcd}
	\end{center}	
	with exact top and bottom triangles. In particular, the top and bottom sequences can be extended to the right in the usual way, $3$-periodic twisted by $\Sigma$. These sequences become injective resolutions of $A$ and $SA'$ in $\modulesfp{\C T}$ via the Yoneda inclusion. By the long exact homology sequence, we can actually do the same with the middle sequence, despite it need not be an exact triangle. This yields a resolution of $B$.
	
	Now, we consider the following induced diagram in $\modulesfp{\C T}$ where, abusing notation, we identify each object in $\C T$ with its Yoneda image in $\modulesfp{\C T}$,
	\begin{center}
		\begin{tikzcd}[column sep=13mm, ampersand replacement=\&]
			A'\arrow[r, hook]\arrow[d, hook]\&X'\arrow[r,"f'"]\arrow[d, hook]\&
			Y'\arrow[r,"f'"]	\arrow[d, hook]
			\&C_{f'}\arrow[d, hook]\arrow[r, two heads]\&\Sigma A'\arrow[d, hook]
			\\
			B\arrow[r, hook]\arrow[d, two heads]\&Y\oplus X'\arrow[d, two heads]\arrow[r, "
			{\left(
				\begin{smallmatrix}
					-i&0\\
					h_2&f'
				\end{smallmatrix}
				\right)}
			"]\&C_f\oplus Y'\arrow[d, two heads]\arrow[r,"{\left(\begin{smallmatrix}
					-q&0\\h_3&i'
				\end{smallmatrix}\right)}"]\& \Sigma X\oplus C_{f'}\arrow[d, two heads]
			\arrow[r, two heads]\& \Sigma B\arrow[d, two heads]\\
			SA\arrow[r, hook]\&Y\arrow[r,"-i"] \&		C_f\arrow[r,"-q"] \&		\Sigma X \arrow[r, two heads]\& \Sigma SA
		\end{tikzcd}
	\end{center}
	Here, the top and bottom extensions represent $\varphi(\Sigma A')$ and $\varphi(\Sigma SA)$, respectively, since they have been obtained from exact triangles, see Section \ref{heller_section}. We have to show that we can find $h_3$ such that the middle extension represents $\varphi(\Sigma B)$. Since we have checked that we can suitably modify $h_1$ and $h_2$, it suffices to show that we can find a representative of $\varphi(\Sigma B)$
	\[B\hookrightarrow P_2\rightarrow P_1\rightarrow P_0\rightarrow \Sigma B\] fitting into a commutative diagram with exact columns,
	\begin{center}
		\begin{tikzcd}[column sep=13mm, ampersand replacement=\&]
			A'\arrow[r, hook]\arrow[d, hook]\&X'\arrow[r,"f'"]\arrow[d, hook]\&
			Y'\arrow[r,"f'"]	\arrow[d, hook]
			\&C_{f'}\arrow[d, hook]\arrow[r, two heads]\&\Sigma A'\arrow[d, hook]
			\\
			B\arrow[r, hook]\arrow[d, two heads]\&P_2\arrow[d, two heads]\arrow[r]\&P_1\arrow[d, two heads]\arrow[r]\& P_0\arrow[d, two heads]
			\arrow[r, two heads]\& \Sigma B\arrow[d, two heads]\\
			SA\arrow[r, hook]\&Y\arrow[r,"-i"] \&		C_f\arrow[r,"-q"] \&		\Sigma X \arrow[r, two heads]\& \Sigma SA
		\end{tikzcd}
	\end{center}
	since then the middle extension must come from a triangle as above. We can even replace the given representatives of $\varphi(\Sigma A')$ and $\varphi(\Sigma SA)$ with any others. Indeed, using standard arguments from homological algebra, like uniqueness of resolutions, it is easy to check that if we can find this diagram for two given representatives then we can also do it for any others. It is not even important that the middle $\C T$-modules are finitely presented. Note also that any short exact sequence in $\modulesfp{\C T}$ can arise as \[A'\hookrightarrow B\twoheadrightarrow SA.\]
	Hence, below we consider an arbitrary one.	
	
	In order to simplify notation, we move from $\modulesfp{\C T}$ to $\modulesfp{\C T_\Sigma}$ in the rest of this proof (it is equivalent since the former abelian category is the degree $0$ part of the latter graded abelian category). We now recall some notation and facts from the proof of Proposition \ref{Heller_criterion_by_differential} that we also need here.

	Let $M$ be a finitely presented $\C T_\Sigma$-module. The element $\varphi(M)\in \ext_{\C T_{\Sigma}}^{3,-1}(M,M)$ is represented by a morphism \[\tilde\varphi(M)\colon M\otimes_{\C T_{\Sigma}}B_3(\C T_{\Sigma})\To M.\] A representing extension can be obtained by taking push-out along $\tilde\varphi(M)$, as in the following diagram
	\begin{center}
		\begin{tikzcd}[column sep=3.5mm]
			M\otimes_{\C T_{\Sigma}}B_3(\C T_{\Sigma})\arrow[r,"d"]\arrow[d,"\tilde\varphi(M)"'] \arrow[dr, phantom, "\text{\scriptsize push}"]&
			M\otimes_{\C T_{\Sigma}}B_2(\C T_{\Sigma})\arrow[r,"d"]\arrow[d]&
			M\otimes_{\C T_{\Sigma}}B_1(\C T_{\Sigma})\arrow[r,"d"]\arrow[d, equal]&
			M\otimes_{\C T_{\Sigma}}B_0(\C T_{\Sigma})\arrow[r,"\epsilon"]\arrow[d, equal]&
			M\arrow[d, equal]\\
			M\arrow[r, hook]&
			P_M\arrow[r]&
			M\otimes_{\C T_{\Sigma}}B_1(\C T_{\Sigma})\arrow[r]&
			M\otimes_{\C T_{\Sigma}}B_0(\C T_{\Sigma})\arrow[r]&
			M
		\end{tikzcd}
	\end{center}
	Here, in the top row, $d$ actually means $\operatorname{id}_M\otimes d$, where $d$ is the bar complex differential, and similarly for the augmentation $\epsilon$. 
	
	Given a morphism $g\colon M\rightarrow N$ in $\modulesfp{\C T_\Sigma}$, the cocycle condition \[(-1)^{\abs{g}}g\cdot \varphi(M)=\varphi(N)\cdot g\]
	is realized by the existence of
	\[\psi(g)\colon M\otimes_{\C T_{\Sigma}}B_2(\C T_{\Sigma})\To N,\]
	which is $k$-linear in $g$, such that 
	\[\psi(g)(d\otimes \operatorname{id})=(-1)^{\abs{g}}g\tilde \varphi(M)-\tilde\varphi(N) (g\otimes\operatorname{id}).\]	
	This allows the definition of a morphism of extensions 
	\begin{center}
		\begin{tikzcd}
			M\arrow[r, hook]\arrow[d,"(-1)^{\abs{g}}g"']&
			P_M\arrow[r]\arrow[d,"P_g"]&
			M\otimes_{\C T_{\Sigma}}B_1(\C T_{\Sigma})\arrow[r]\arrow[d,"g\otimes\operatorname{id}"]&
			M\otimes_{\C T_{\Sigma}}B_0(\C T_{\Sigma})\arrow[r]\arrow[d,"g\otimes\operatorname{id}"]&
			M\arrow[d,"g"]\\
			N\arrow[r, hook]&
			P_N\arrow[r]&
			N\otimes_{\C T_{\Sigma}}B_1(\C T_{\Sigma})\arrow[r]&
			N\otimes_{\C T_{\Sigma}}B_0(\C T_{\Sigma})\arrow[r]&
			N
		\end{tikzcd}
	\end{center}
	where the morphism $P_g$ is defined by applying the universal property of a push-out to the following diagram
	\begin{center}
		\begin{tikzcd}
			M\otimes_{\C T_{\Sigma}}B_3(\C T_{\Sigma})\arrow[r,"\operatorname{id}_M\otimes d"]\arrow[d,"\tilde\varphi(M)"] \arrow[dr, phantom, "\text{\scriptsize push}"]\arrow[d]&
			M\otimes_{\C T_{\Sigma}}B_2(\C T_{\Sigma})\arrow[d]\arrow[ddr, bend left=10, "\parbox{18mm}{\scriptsize sum of the two dashed paths}", near end]\arrow[r,"g\otimes\operatorname{id}", dashed]\arrow[ldd, dashed, bend left=10, "\psi(g)"]&
			N\otimes_{\C T_{\Sigma}}B_2(\C T_{\Sigma})\arrow[dd, bend left=75, dashed]\\
			M\arrow[r, hook, crossing over]\arrow[d,"(-1)^{\abs{g}}g"']&P_M\arrow[dr, dotted, "P_f"]&\\
			N\arrow[rr, hook]&
			&
			P_N
		\end{tikzcd}
	\end{center}

	If $d_2(\varphi)=0$ then, for any pair of composable morphisms in $\modulesfp{\C T_\Sigma}$,
	\[L\stackrel{f}{\To}M\stackrel{g}{\To }N,\]
	there exists a morphism
	\[\zeta(g\otimes f)\colon L\otimes B_1(\C T_\Sigma)\longrightarrow N,\]
	bilinear in $f$ and $g$, such that
	\[\zeta(g\otimes f)(\operatorname{id}\otimes d)=(-1)^{\abs{g}}g\psi(f)-\psi(gf)+\psi(g)(f\otimes\operatorname{id}).\]
	In particular, $P_gP_f-P_{gf}$ is the composite
	\[P_L\longrightarrow L\otimes B_1(\C T_\Sigma)\stackrel{\zeta(g\otimes f)}\longrightarrow N\hookrightarrow P_N.\]
	
	Assume $g\colon M\twoheadrightarrow N$ is surjective. Then we can find a lift \[\zeta'\colon L\otimes B_1(\C T_\Sigma)\longrightarrow M\] of $\zeta(g\otimes f)$ along $g$, and if we define $\delta$ as 
	\[(-1)^{\abs{g}+1}\delta\colon P_L\longrightarrow L\otimes B_1(\C T_\Sigma)\stackrel{\zeta'}\longrightarrow M\hookrightarrow P_M,\]
	then
	\[P_g(P_f+\delta)=P_{gf}.\]
	Moreover, 
	\begin{center}
		\begin{tikzcd}
			L\arrow[r, hook]\arrow[d,"(-1)^{\abs{f}}f"']&
			P_L\arrow[r]\arrow[d,"P_f+\delta"]&
			L\otimes_{\C T_{\Sigma}}B_1(\C T_{\Sigma})\arrow[r]\arrow[d,"f\otimes\operatorname{id}"]&
			L\otimes_{\C T_{\Sigma}}B_0(\C T_{\Sigma})\arrow[r]\arrow[d,"f\otimes\operatorname{id}"]&
			L\arrow[d,"f"]\\
			M\arrow[r, hook]&
			P_M\arrow[r]&
			M\otimes_{\C T_{\Sigma}}B_1(\C T_{\Sigma})\arrow[r]&
			M\otimes_{\C T_{\Sigma}}B_0(\C T_{\Sigma})\arrow[r]&
			M
		\end{tikzcd}
	\end{center}
	is still a map of extensions.
	
	We claim that, if 
	\[L\stackrel{f}{\hookrightarrow}M\stackrel{g}{\twoheadrightarrow}N\]
	is a short exact sequence, then so are the colums of 
	\begin{center}
		\begin{tikzcd}
			L\arrow[r, hook]\arrow[d, hook,"(-1)^{\abs{f}}f"']&
			P_L\arrow[r]\arrow[d,"P_f+\delta"]&
			L\otimes_{\C T_{\Sigma}}B_1(\C T_{\Sigma})\arrow[r]\arrow[d, hook,"f\otimes\operatorname{id}"]&
			L\otimes_{\C T_{\Sigma}}B_0(\C T_{\Sigma})\arrow[r]\arrow[d, hook,"f\otimes\operatorname{id}"]&
			L\arrow[d, hook,"f"]\\
			M\arrow[r, hook]\arrow[d, two heads,"(-1)^{\abs{g}}g"']&
			P_M\arrow[r]\arrow[d,"P_g"]&
			M\otimes_{\C T_{\Sigma}}B_1(\C T_{\Sigma})\arrow[r]\arrow[d, two heads,"g\otimes\operatorname{id}"]&
			M\otimes_{\C T_{\Sigma}}B_0(\C T_{\Sigma})\arrow[r]\arrow[d, two heads,"g\otimes\operatorname{id}"]&
			M\arrow[d, two heads,"g"]\\
			N\arrow[r, hook]&
			P_N\arrow[r]&
			N\otimes_{\C T_{\Sigma}}B_1(\C T_{\Sigma})\arrow[r]&
			N\otimes_{\C T_{\Sigma}}B_0(\C T_{\Sigma})\arrow[r]&
			N
		\end{tikzcd}
	\end{center}
	The only column where the claim is not obvious is the second one. We know at least that $P_g(P_f+\delta)=P_{gf}=P_0=0$. If we write $Q_L,Q_M,Q_N$ for the kernels of the horizontal arrows between the third and fourth columns, the long exact homology sequence yields a short exact sequence
	\[Q_L\hookrightarrow Q_M\twoheadrightarrow Q_N.\]
	These are also the images of the horizontal arrows between the second and third columns. Hence the previous short exact sequence fits in a commutative diagram
	\begin{center}
	\begin{tikzcd}
		L\arrow[r, hook]\arrow[d, hook,"(-1)^{\abs{f}}f"']&
		P_L\arrow[r, two heads]\arrow[d,"P_f+\delta"]&
		Q_L\arrow[d, hook]\\
		M\arrow[r, hook]\arrow[d, two heads,"(-1)^{\abs{g}}g"']&
		P_M\arrow[r,two heads]\arrow[d,"P_g"]&
		Q_M\arrow[d, two heads]\\
		N\arrow[r, hook]&
		P_N\arrow[r]&
		Q_N
	\end{tikzcd}
	\end{center}
	Applying the snake lemma to the top and bottom maps of short exact sequences we see that $P_f+\delta$ is injective and $P_g$ is surjective. Moreover, since the middle column composes to zero, the diagram and the snake lemma also yield a map of extensions
		\begin{center}
		\begin{tikzcd}
			L\arrow[r, hook]\arrow[d, equal]&
			P_L\arrow[r, two heads]\arrow[d]&
			Q_L\arrow[d, equal]\\
			L\arrow[r, hook]&
			\ker P_g\arrow[r]&
			Q_L
		\end{tikzcd}
	\end{center}
	which concludes the proof of exactness since the middle vertical arrow is necessarily an isomorphism by the five lemma.	
\end{proof}

Following \cite{bondal_enhanced_1991}, we say that a DG- or $A$-infinity category $\C C$ is \emph{pre-triangulated} if $H^*(\C C)$ is weakly stable, i.e.~up to equivalence it is of the form $\C T_\Sigma$, and the usual Massey product in the cohomology $H^*(\C C)\simeq\C T_{\Sigma}$ induces a triangulated structure on $H^0(\C C)\simeq\C T$ with suspension $\Sigma$. Over a field, the \emph{universal Massey product} of a DG- or $A$-infinity category $\C C$,
\[\{m_3\}\in \hh{3,-1}{H^*(\C C),H^*(\C C)},\]
is the universal Massey product of any minimal model (it does not depend on the choice). 

The following result is a direct consequence of Proposition \ref{Heller_criterion_by_differential} and Theorem \ref{octahedral_theorem}.

\begin{corollary}\label{pretriangulated}
	Let $\C C$ be a DG- or $A$-infinity category over a field $k$ such that idempotents in $H^0(\C C)$ split. The following statements are equivalent:
	\begin{itemize}
		\item $\C C$ is pre-triangulated.
		\item $H^*(\C C)$ is weakly stable, $\modulesfp{H^0(\C C)}$ is Frobenius abelian, and the image of the universal Massey product $\{m_3\}\in \hh{3,-1}{H^*(\C C), H^*(\C C)}$ along the edge morphism \[\hh{3,-1}{H^*(\C C), H^*(\C C)}\longrightarrow \hh{0,-1}{\modulesfp{H^*(\C C)},\ext_{H^*(\C C)}^{3,\ast}}\] in Theorem \ref{theorem_edge_morphism} is a unit in the bigraded algebra \[\hh{0,*}{\modulesfp{H^*(\C C)},\widehat{\ext}_{H^*(\C C)}^{\star,\ast}}.\]
	\end{itemize}
\end{corollary}

\section{The topological case}\label{topological_section}

In this section we move to a non-additive setting. Let $\operatorname{Set}$ be the category of \emph{graded sets} $X=\{X_n\}_{n\in\mathbb Z}$. We endow it with the closed symmetric monoidal structure defined as \[(X\boxtimes Y)_n=\coprod_{n=p+q}X_p\times Y_q.\] In this section, a \emph{graded category} is a category enriched in $\operatorname{Set}$. We will also use graded liner categories, always defined over $\mathbb Z$. Here we will specify when a given (graded) category is linear.

The obvious forgetful functor
\[\modules{\mathbb Z}\longrightarrow\operatorname{Set}\]
is lax monoidal. Given two graded abelian groups $A$ and $B$, the natural map
\[A\boxtimes B\longrightarrow A\otimes B\]
is given by the universal bilinear maps $A_p\times B_q\rightarrow A_p\otimes B_q$. 
The forgetful functor has a left adjoint, the \emph{free graded abelian group} functor
\[\operatorname{Set}\longrightarrow\modules{\mathbb Z}\colon X\mapsto\mathbb Z(X),\]
which is strict monoidal. Neither is symmetric because there is no way to encode the Koszul sign rule in $\operatorname{Set}$. Nevertheless, that is sufficient to functorially define the graded linear category $\mathbb Z\C C$ associated to a graded category $\C C$, that we call \emph{linearization}. It has the same objects as $\C C$ and morphism objects \[(\mathbb Z\C C)(X,Y)=\mathbb Z\C C(X,Y),\]
and composition is defined by
\[\mathbb Z\C C(Y,Z)\otimes \mathbb Z\C C(X,Y)\cong \mathbb Z(\C C(Y,Z)\boxtimes \C C(X,Y))\longrightarrow \mathbb Z\C C(X,Z).\]
Linearization is the left adjoint of the forgetful functor from graded linear categories to graded categories. For the sake of simplicity, the linearization of a graded functor between graded categories $F\colon \C C\rightarrow\C D$ will also be denoted by $F\colon \mathbb Z\C C\rightarrow\mathbb Z\C D$. The lack of compatibility with the symmetry constraint implies that the linearization functor does not take tensor products of graded categories to tensor products of graded linear categories, but this will be irrelevant because we will only consider the latter. 

The \emph{cohomology} of a graded category $\C C$ with coefficients in a $\mathbb Z\C C$-bimodule $M$ is defined as
\[H^{\star,*}(\C C,M)=\hh{\star,*}{\mathbb Z\C C, M}.\]
A $\mathbb Z\C C$-bimodule can also be described as a family of graded abelian groups $M(X,Y)$ indexed by pairs of objects $X,Y$ in $\C C$
and, for each four objects $X,X',Y,Y'$ in $\C C$, a degree $0$ map of graded sets
\begin{align*}
\C C(Y,Y')\boxtimes M(X,Y)\boxtimes \C C(X',X)&\longrightarrow M(X',Y'),\\
(g, x, f)&\;\mapsto\; g\cdot x\cdot f,
\end{align*}
or equivalently, degree $0$ maps of graded sets, $p,n,q\in\mathbb Z$,
\begin{align*}
\C C^p(Y,Y')\times M^n(X,Y)\times \C C^q(X',X)&\longrightarrow M^{p+n+q}(X',Y'),\\
(g, x, f)&\;\mapsto\; g\cdot x\cdot f,
\end{align*}
which are linear in $x$ and satisfy the usual associativity and unit conditions. 

The previous cohomology of graded categories satisfies the same functoriality properties as Hochschild cohomology, compare \cite{muro_functoriality_2006}. 
The ungraded version was also considered in \cite{mitchell_rings_1972}. The cochain complex $\hc{\star, \ast}{\mathbb Z\C C,M}$ defining $H^{\star,*}(\C C,M)$ can also be described as
\[\hc{n,\ast}{\mathbb Z\C C,M}=\prod_{X_0,\dots, X_n} \hom^*_{\operatorname{Set}}(\C C(X_1,X_{0})\times\cdots\times\C C(X_n,X_{n-1}), M(X_n,X_0)),\]
where $\hom^*_{\operatorname{Set}}$ is the inner $\hom$ in graded sets. The differential is given by \eqref{hochschild_differential}.

The new cohomology theory makes sense even if $\C C$ is additive. In this case, it is (or deserves to be called) the \emph{topological Hochschild cohomology} or \emph{Mac Lane cohomology} of $\C C$, see \cite{pirashvili_mac_1992, jibladze_cohomology_1991}. Moreover, the obvious linear functor $\mathbb Z\C C\rightarrow\C C$ induces \emph{comparison morphisms}
\[\hh{\star,*}{\C C,M}\To H^{\star,*}(\C C,M)\]
for any $\C C$-bimodule $M$. This morphism is an isomorphism for $\star=0$ since the end of $M$ regarded as a $\C C$-bimodule is the same as if we regard it as a $\mathbb Z\C C$-bimodule, because $\mathbb Z\C C\rightarrow\C C$ is full. In particular, we can replace $HH^{0,*}$ with $H^{0,*}$ in all results of Sections \ref{heller_section} and \ref{toda_brackets_section}. At the level of cochains, the comparison morphism
\[\hc{\star,*}{\C C,M}\hookrightarrow \hc{\star,*}{\mathbb Z\C C,M}\]
is the inclusion of multilinear cochains.

Given a plain ungraded category $\C T$ and an automorphism $\Sigma\colon \C T\rightarrow\C T$, the graded category $\C T_\Sigma$ also makes sense in the non-additive context, and it is compatible with linearization, $(\mathbb Z\C T)_{\Sigma}=\mathbb Z(\C T_\Sigma)$. In particular, if $M$ is an ungraded $\mathbb Z\C T$-bimodule equipped with an isomorphism $\tau\colon M\cong M(\Sigma,\Sigma)$ 
such that, given
$g\in\C T(Y,Y')$, $x\in M(X,Y)$, and $f\in\C T(X',X)$,
\[\tau(g\cdot x\cdot f)=(\Sigma g)\cdot\tau(x)\cdot(\Sigma f),\]
we can define the $\C T_\Sigma$-bimodule $M_\tau$ as in Section \ref{Hochschild_cohomology_of_categories}. Proposition \ref{graded_ungraded_long_exact_sequence} applies, so we have a long exact sequence 
\begin{equation}\label{exact_cohomology_sequence_top}
	\begin{tikzcd}[row sep=5mm]
		\vdots\arrow[d]\\
		\hbw{n,*}{\C T_{\Sigma},M_{\tau}}\arrow[d, "i^*"]\\
		\hbw{n,*}{\C T,M_{\tau}}\arrow[d, "1-\tau_*^{-1}\Sigma^*"]\\ 
		\hbw{n,*}{\C T,M_{\tau}}\arrow[d]\\
		\hbw{n+1,*}{\C T_{\Sigma},M_{\tau}}\arrow[d]\\
		\vdots
	\end{tikzcd}
\end{equation}
This, or rather the proof of Proposition \ref{graded_ungraded_long_exact_sequence}, shows that $H^{\star,*}(\C T_{\Sigma},M_\tau)$ coincides with the translation cohomology of the pair $(\Sigma,\tau)$ defined in \cite{baues_homotopy_2007} and extensively used in \cite{baues_cohomologically_2008}. 

If $\C M$ is a stable model category, its homotopy category $\C T=\operatorname{Ho}\C M$ is triangulated \cite{hovey_model_1999} with the well known suspension functor $\Sigma\colon \C T\rightarrow \C T$. Moreover, objects, maps, and tracks (i.e.~homotopy classes of homotopies relative to the boundary) define a topological analogue of the universal Massey product, that we call \emph{universal Toda bracket} \cite[Remark 5.9]{baues_homotopy_2007},
\[\langle\C M\rangle\in H^{3,-1}(\C T_\Sigma,\C T_\Sigma),\]
save for the fact that the category $\C T$ may be big, but we can replace $\C T$ with any small triangulated subcategory of $\operatorname{Ho}\C M$, whose cohomology is then well defined. Hence, any topological enhancement of a small triangulated category defines a universal Toda bracket. 

A representing cocycle $m_3$ for $\langle\C M\rangle$ is formally a ternary operation as in the introduction, except for the fact that it need not be multilinear. Universal Toda brackets have been extensively studied in the ungraded setting, i.e.~the image of $\langle\C M\rangle$ along the morphism 
\[H^{3,-1}(\C T_\Sigma,\C T_\Sigma)\longrightarrow H^{3,-1}(\C T,\C T_\Sigma)\cong H^3(\C T,\C T(\Sigma,-))\]
fitting in a long exact sequence as above.
They indeed determine all homotopically defined Toda brackets in $\operatorname{Ho}\C M$ \cite{baues_cohomology_1989}, which characterize its triangulated structure in the way explained in the introduction. 

The spectral sequence in Proposition \ref{spectral_sequence} is also defined in the current non-additive context but the proof, although similar, needs a couple of significant modifications.  

\begin{proposition}\label{spectral_sequence_top}
	If $\C T$ is a small additive category such that $\modulesfp{\C T}$ is abelian and $\Sigma\colon\C T\rightarrow\C T$ is an automorphism, there is a first quadrant cohomological spectral sequence of graded abelian groups
	\[E_2^{p,q}=H^{p,*}(\modulesfp{\C T_{\Sigma}},\ext_{\C T_{\Sigma}}^{p,*})\Longrightarrow H^{p+q,*}(\C T_{\Sigma},\C T_{\Sigma}).\]
\end{proposition}

\begin{proof}
	Now the spectral sequence is associated to the following bicomplex of graded modules $C^{\star,\bullet}$,
	\[\hom_{\mathbb Z\modulesfp{\C T_{\Sigma}}^{\env}}^*(B_{\star}(\mathbb Z\modulesfp{\C T_{\Sigma}}),\hom_{\C T_{\Sigma}}^*(\mathbb Z(-)\otimes_{\mathbb Z\C T_{\Sigma}}B_{\bullet}(\mathbb Z\C T_{\Sigma})\otimes_{\mathbb Z\C T_{\Sigma}}\C T_{\Sigma},-)).\]
	Here we use the \emph{forceful linearization} of right $\C T_\Sigma$-modules, which is a non-additive functor \[\mathbb Z(-)\colon \modules{\C T_{\Sigma}}\longrightarrow\modules{\mathbb Z\C T_{\Sigma}}\]
	defined as follows. Given a right $\C T_{\Sigma}$-module $M$, $\mathbb Z(M)(X)=\mathbb Z(M(X))$ for any object $X$ in $\C T_\Sigma$ and the right action of $\mathbb Z\C T_{\Sigma}$ is given by
	\[\begin{split}
	\mathbb Z(M)(X)\otimes\mathbb Z\C T_{\Sigma}(X',X)&\cong\mathbb Z(M(X)\boxtimes \C T_{\Sigma}(X',X))\\&\rightarrow \mathbb Z(M(X)\otimes \C T_{\Sigma}(X',X))\\&\rightarrow \mathbb Z(M)(X'). 
	\end{split}\]
	Moreover, given a morphism of right $\C T_{\Sigma}$-modules $f\colon M\rightarrow N$ of any degree, the induced morphism $\mathbb Z(f)\colon \mathbb Z(M)\rightarrow \mathbb Z(N)$ is given by the free abelian group homomorphisms $\mathbb Z(f)(X)\colon \mathbb Z(M)(X)\rightarrow \mathbb Z(N)(X)$ defined by $f(X)\colon M(X)\rightarrow N(X)$ on the bases, where $X$ is any object in $\C T_\Sigma$. 
	We remark for later use that the forceful linearization functor preserves colimits, since the free grade dabelian group functor, which is a left adjoint, preserves colimits, and colimits of right modules are computed pointwise.
	
	An element of $C^{p,q}$ is the same a family of maps of graded sets
	\[\prod_{i=1}^p\hom^*_{\C T_\Sigma}(M_{i},M_{j-1})
	\times M_p(X_0)\times
	\prod_{j=1}^q\C T_\Sigma(X_{j},X_{j-1})\longrightarrow M_0(X_q)\]
	indexed by all sequences of objects $M_0,\dots, M_p$ in $\modulesfp{\C T_\Sigma}$ and $X_0,\dots, X_q$ in $\C T_\Sigma$. With this description, the horizontal and vertical differentials are again \eqref{differentials_bicomplex}. 
	
	We must identify the $E_2$-term and the target of the spectral sequence. First, observe that for each finitely presented right $\C T_{\Sigma}$-module $M$, the complex of projective right $\C T_{\Sigma}$-modules
	\[\mathbb Z (M)\otimes_{\mathbb Z\C T_{\Sigma}}B_{\bullet}(\mathbb Z\C T_{\Sigma})\otimes_{\mathbb Z\C T_{\Sigma}}\C T_{\Sigma}\]
	is the standard complex computing the André--Quillen homology $H_*(M,y)$ of $M$ with coefficients in the Yoneda inclusion $y\colon \C T_{\Sigma}\hookrightarrow\modulesfp{\C T_{\Sigma}}$ in the following graded \emph{catégorie avec modèles munis de coefficients}, using André's terminology, see \cite[Chapitre I]{andre_methode_1967}, \[\modulesfp{\C T_{\Sigma}}\supset\C T_{\Sigma}\stackrel{y}{\longrightarrow}\modulesfp{\C T_{\Sigma}}.\] 
	Indeed, 
	\begin{multline*}
	\mathbb Z (M)\otimes_{\mathbb Z\C T_{\Sigma}}B_{n}(\mathbb Z\C T_{\Sigma})\otimes_{\mathbb Z\C T_{\Sigma}}\C T_{\Sigma}\\
	=\bigoplus_{X_0,\dots, X_n} \mathbb Z(M(X_0))\otimes\mathbb Z\C T_{\Sigma}(X_1,X_{0})\otimes\cdots\otimes\mathbb Z\C T_{\Sigma}(X_n,X_{n-1})\otimes\C T_{\Sigma}(-,X_n)\\
	=\bigoplus_{M\leftarrow X_0\leftarrow \cdots\leftarrow  X_n}\C T_{\Sigma}(-,X_n).
	\end{multline*}

	Unlike in Proposition \ref{spectral_sequence}, we here restrict the statement to graded linear categories of the form $\C T_{\Sigma}$ with $\C T$ additive and $\modulesfp{\C T}$ abelian. We will use this fact now. More precisely, in order to apply a result of André we will use that, under these hypotheses, any object in $\modulesfp{\C T_{\Sigma}}$ has a projective resolution by objects in $\C T_{\Sigma}$.

	Since $y$ is tautologically the restriction of the identity functor in $\modulesfp{\C T_{\Sigma}}$ to $\C T_\Sigma$, and the identity functor is exact, then $H_*(M,y)$ is naturally $M$ concentrated in degree $0$, see \cite[Proposition 13.2]{andre_methode_1967}. Therefore the complex $\mathbb Z (M)\otimes_{\mathbb Z\C T_{\Sigma}}B_{\bullet}(\mathbb Z\C T_{\Sigma})\otimes_{\mathbb Z\C T_{\Sigma}}\C T_{\Sigma}$ is a projective resolution of $M$. In particular, since $\mathbb Z\modulesfp{\C T_{\Sigma}}$ is locally free by definition, the $E_2$-term of the first-vertical-then-horizontal spectral sequence is as in the statement. 
	
	André's result would in principle require that $\C T_\Sigma$ identified with the full subcategory of projectives in $\modulesfp{\C T_\Sigma}$ under the Yoneda inclusion. That would be true if $\C T$ were idempotent complete, which is a harmless common assumption in this paper. Nevertheless, it suffices that we can form a simplicial resolution of any finitely presented right $\C T_\Sigma$-module $M$ with objects in $\C T_\Sigma$, and by the Dold--Kan equivalence this follows from the existence of a projective resolution of $M$ by objects in $\C T_\Sigma$.
	
	In order to compute the target of the previous spectral sequence of the bicomplex $C^{\star,\bullet}$ we consider the other one, exactly as in the proof of Proposition \ref{spectral_sequence}. 
	
	The $\mathbb Z\modulesfp{\C T_{\Sigma}}$-bimodule $\hom_{\C T_{\Sigma}}^*(\mathbb Z(-)\otimes_{\mathbb Z\C T_{\Sigma}}B_{n}(\mathbb Z\C T_{\Sigma})\otimes_{\mathbb Z\C T_{\Sigma}}\C T_{\Sigma},-))$ sends $M$ and $N$ to 
	\[\prod_{X_1,\dots,X_n}\hom^*_{\operatorname{Set}}(M(X_0)\times\C T_{\Sigma}(X_1,X_{0})\times\cdots\times\C T_{\Sigma}(X_n,X_{n-1}),N(X_n)),\]
	hence, it is a product of $\mathbb Z\modulesfp{\C T_{\Sigma}}$-bimodules $D_{X,Y}$ of the form $D_{X,Y}(M,N)=\hom_{\operatorname{Set}}^*(M(X),N(Y))$, where $X$ and $Y$ are fixed objects in $\C T_{\Sigma}$. The cohomology of $\modulesfp{\C T_{\Sigma}}$ with coefficients in such a $D_{X,Y}$ is concentrated in degree $0$ by the graded version of \cite[Lemma 3.9]{jibladze_cohomology_1991} (which actually follows from the ungraded original version via the graded-ungraded long exact sequence). 	Hence, the inclusion of the $0$-dimensional horizontal cohomology in $C^{\star,\bullet}$ is a quasi-isomorphism (with the total complex). This $0$-dimensional horizontal cohomology is the end of the cochain complex of $\mathbb Z\modulesfp{\C T_{\Sigma}}$-bimodules $\hom_{\C T_{\Sigma}}^*(\mathbb Z(-)\otimes_{\mathbb Z\C T_{\Sigma}}B_{\bullet}(\mathbb Z\C T_{\Sigma})\otimes_{\mathbb Z\C T_{\Sigma}}\C T_{\Sigma},-))$.
	By the extension-restriction of scalars adjunction, this complex coincides with $\hom_{\mathbb Z\C T_{\Sigma}}^*(\mathbb Z(-)\otimes_{\mathbb Z\C T_{\Sigma}}B_{\bullet}(\mathbb Z\C T_{\Sigma}),-))$. The end is, dimensionwise, the graded abelian group of natural transformations from the source to the target regarded as graded functors $\modulesfp{\C T_{\Sigma}}\rightarrow \modules{\C T_{\Sigma}}$. The source preserves colimits, and $\C T_\Sigma\subset\modulesfp{\C T_\Sigma}$ is the inclusion of a dense subcategory \cite[\S5.1]{kelly_basic_2005}, hence the source is the left Kan extension of its restriction along $\C T_\Sigma\subset\modulesfp{\C T_\Sigma}$ \cite[Theorem 5.29]{kelly_basic_2005}, so the end can be computed by restricting to $\C T_\Sigma$. The latter end is the complex $\hom_{\mathbb Z\C T_{\Sigma}^{\env}}^*(B_\bullet(\mathbb Z\C T_{\Sigma}),\C T_{\Sigma})$, whose cohomology is the claimed target of the spectral sequence, hence we are done. An explicit quasi-isomorphism $\xi\colon \hom_{\mathbb Z\C T_{\Sigma}^{\env}}^*(B_\bullet(\mathbb Z\C T_{\Sigma}),\C T_{\Sigma})\hookrightarrow C^{\star,\bullet}$ is defined as in \eqref{quasi_iso}.
\end{proof}

Now we can state the analogue of Theorem \ref{theorem_edge_morphism}.

\begin{theorem}\label{theorem_edge_morphism_top}
	Let $\C T$ be an idempotent complete triangulated category with suspension $\Sigma$ and a topological enhancement. The edge morphism 
	\[\hbw{3,-1}{\C T_{\Sigma}, \C T_{\Sigma}}\To \hbw{0,-1}{\modulesfp{\C T_{\Sigma}},\ext_{\C T_{\Sigma}}^{3,\ast}}
	\]
	of the spectral sequence in Proposition \ref{spectral_sequence_top} takes the universal Toda bracket of the enhancement to the Toda bracket of the triangulated structure.
\end{theorem}

The proof is exactly the same as that of Theorem \ref{theorem_edge_morphism}, replacing the reference to Proposition \ref{spectral_sequence} with Proposition \ref{spectral_sequence_top}.

Despite this section's non-additive cohomology is based on cochains which are not multilinear, in certain cases we can compute it using cochains which at least vanish when evaluated at zero maps. More precisely, let $\C C$ be a graded category with a zero object $0$. A $\mathbb Z\C C$-bimodule $M$ is \emph{zero-trivial} if it vanishes when evaluated at the zero object of $\C C$ at any slot, $M(0,-)=0=M(-,0)$. If $\C C$ is additive, any $\C C$-bimodule regarded as a $\mathbb Z\C C$-bimodule is zero-trivial. A cochain $\varphi\in \hc{\star, \ast}{\mathbb Z\C C,M}$ is \emph{zero-normalized} if it vanishes whenever we put a trivial morphism in one of the slots $\varphi(\dots,0,\dots)=0$. The inclusion of the subcomplex $\hz{\star, \ast}{\mathbb Z\C C,M}\subset \hc{\star, \ast}{\mathbb Z\C C,M}$ consisting of zero-normalized cochains is a quasi-isomorphism when $M$ is zero-trivial, actually a chain homotopy  equivalence, compare \cite[Theorem 1.10 and Appendix B]{baues_cohomology_1989}. The subcomplex of normalized cochains is better explained as follows. When $\C C$ has a zero object we can construct its \emph{zero-normalized linearization} $\breve{\mathbb Z}\C C$, with the same objects as $\C C$ and where $(\breve{\mathbb Z}\C C)(X,Y)$ is obtained from $\mathbb Z\C C(X,Y)$ by quotienting out the zero maps $0\in\C C(X,Y)$. There is an obvious projection linear functor $\mathbb Z\C C\rightarrow \breve{\mathbb Z}\C C$ which is the identity on objects. A zero-trivial $\mathbb Z\C C$-bimodule is the same as a $\breve{\mathbb Z}\C C$-bimodule $M$, and $\hz{\star, \ast}{\mathbb Z\C C,M}=\hc{\star, \ast}{\breve{\mathbb Z}\C C,M}$. The inclusion of zero-normalized cochains is induced by the projection $\mathbb Z\C C\rightarrow \breve{\mathbb Z}\C C$.

Given a graded linear category $\C C$ such that $\modulesfp{\C C}$ is Frobenius abelian. The same formula as in Definition \ref{kappa}, using now zero-normalized cochains, yields graded abelian group morphisms, $p\geq 1$, $q\in\mathbb Z$,
\[\kappa \colon \hbw{p+1,*}{\modulesfp{\C C },\widehat{\ext}_{\C C}^{q,*}}
\To
\hbw{0,*}{\modulesfp{\C C },\widehat{\ext}_{\C C}^{p+q,*}}.\]
The use of zero-normalized cochains ensures that the definition does not depend on the choice of representing cocycles.

The analogue of Proposition \ref{Heller_criterion_by_differential} holds in our current non-additive setting.

\begin{proposition}\label{Heller_criterion_by_differential_top}
	If $\C T$ is a small ungraded additive category such that $\modulesfp{\C T}$ is Frobenius abelian and $\Sigma\colon \C T\rightarrow\C T$ is an automorphism, then the	 set of Puppe triangulated structures on $\C T$ with suspension functor $\Sigma$ is in bijection with the units of the bigraded ring $\hbw{0,*}{\modulesfp{\C T_\Sigma},\widehat{\ext}_{\C T_{\Sigma}}^{\bullet,*}}$ lying in the kernel of the composite
	\begin{center}
		\begin{tikzcd}
		\hbw{0,-1}{\modulesfp{\C T_\Sigma},\ext_{\C T_{\Sigma}}^{3,*}}
		\arrow[d,"d_2"]\\
		\hbw{2,-1}{\modulesfp{\C T_\Sigma},\ext_{\C T_{\Sigma}}^{2,*}}
		\arrow[d,"\kappa"]\\
		\hbw{0,-1}{\modulesfp{\C T_\Sigma},\ext_{\C T_{\Sigma}}^{3,*}}
		\end{tikzcd}
	\end{center}
	where $d_2$ is a second differential in the spectral sequence of Proposition \ref{spectral_sequence_top}.
\end{proposition}

The proof of this result is essentially the same as the proof of Proposition \ref{Heller_criterion_by_differential}. We only need to replace the resolution $M\otimes_{\C T_\Sigma}B_\bullet(\C T_\Sigma)$ of $M$ used previously with $\mathbb Z (M)\otimes_{\mathbb Z\C T_{\Sigma}}B_{\bullet}(\mathbb Z\C T_{\Sigma})\otimes_{\mathbb Z\C T_{\Sigma}}\C T_{\Sigma}$, that we use in the proof of Proposition \ref{spectral_sequence_top}. We did not really use in an essential way the multilinearity of cocycles therein. We only used that cocycles vanish when one variable is the trivial morphism. Hence, we must restrict to zero-normalized cochains, i.e.~we must actually use $\breve{\mathbb Z}\C C$ instead of $\mathbb Z\C C$ and we must define a zero-normalized forceful linearization $\breve{\mathbb Z}(M)$ of right $\C T_\Sigma$-modules $M$ where $\breve{\mathbb Z}(M)(X)$
is obtained from $\mathbb Z(M(X))$ by quotienting out zeroes $0\in M(X)$.

Theorem \ref{octahedral_theorem} (the sufficient condition for the octahedral axiom) also holds true for the spectral sequence in Proposition \ref{spectral_sequence_top}, so triangulated categories with a universal Toda bracket satisfy the octahedral axiom (this was already checked in \cite{baues_cohomologically_2008}). 

\begin{theorem}\label{octahedral_theorem_top}
	If $\C T$ is a small ungraded idempotent complete additive category such that $\modulesfp{\C T}$ is Frobenius abelian, $\Sigma\colon \C T\rightarrow\C T$ is an automorphism, and $\varphi\in \hbw{0,-1}{\modulesfp{\C T_\Sigma},\ext_{\C T_{\Sigma}}^{3,*}}$ corresponds to a Puppe triangulated structure on $\C T$ with suspension $\Sigma$ such that $d_2(\varphi)=0$, then this Puppe triangulated structure satisfies the octahedral axiom.
\end{theorem}

Again, the proof of this theorem is not much different to its linear version over a field, up to the previous replacement of resolutions.

The topological analogue of DG- or $A$-infinity categories are \emph{spectral categories}, i.e.~categories enriched in any closed symmetric monoidal category of spectra, such as symmetric spectra. Given a spectral category $\C C$, we can take stable homotopy groups on morphism spectra and form a graded linear category $\pi_*\C C$. In the same way as a DG-category induces Massey products in cohomology, a spectral category $\C C$ induces Toda brackets in stable homotopy, hence if $\pi_*\C C \simeq \C T_\Sigma$ is weakly stable then $(\C T,\Sigma)$ is endowed with a stable Toda bracket. We say that $\C C$ is \emph{pre-triangulated} if in the previous circumstances the stable Toda bracket induces a triangulated structure on $(\C T,\Sigma)$. This stable Toda bracket comes from a universal Toda bracket. Indeed, we can consider the stable model category $\modules{\C C}$ of right $\C C$-modules, and the Yoneda inclusion $\C C\subset\modules{\C C}$ gives rise to a full inclusion $\pi_0\C C\subset\operatorname{Ho}\modules{\C C}$. If $\Sigma$ is the suspension functor in $\operatorname{Ho}\modules{\C C}$ and $\pi_*\C C$ is weakly stable then $\pi_*\C C=(\pi_0\C C)_{\Sigma}$, so it inherits the universal Toda bracket. Up to idempotent completion, our pre-triangulated spectral categories coincide with Tabuada's triangulated spectral categories \cite[Definition 5.1]{tabuada_matrix_2010}. The topological analogue of Corollary \ref{pretriangulated} is the following result.

\begin{corollary}\label{pretriangulated_top}
	Let $\C C$ be a spectral category such that idempotents in $\pi_0\C C$ split. The following statements are equivalent:
	\begin{itemize}
		\item $\C C$ is pre-triangulated.
		\item $\pi_*\C C$ is weakly stable, $\modulesfp{\pi_0\C C}$ is Frobenius abelian, and the image of the universal Massey product along the edge morphism \[\hbw{3,-1}{\pi_*\C C, \pi_*\C C}\longrightarrow \hbw{0,-1}{\modulesfp{\pi_*\C C},\ext_{\pi_*\C C}^{3,\ast}}\] in Theorem \ref{theorem_edge_morphism_top} is a unit in the bigraded algebra \[\hbw{0,*}{\modulesfp{\pi_*\C C},\widehat{\ext}_{\pi_*\C C}^{\star,\ast}}.\]
	\end{itemize}
\end{corollary}

\section{Example of non-vanishing obstructions}\label{example}

We here illustrate with an example that the obstructions need not vanish. For this, we need a triangulated category without enhancements. There are few known examples, essentially those in  \cite{muro_triangulated_2007}, their non-commutative analogues \cite{dimitrova_triangulated_2009}, and a new recent family of examples over the rationals \cite{rizzardo_k-linear_2018}. These new examples do not have an $A_\infty$-enhancement but they do have an $A_6$-enhancement, so they have a universal Massey product, and therefore the first obstructions described in this paper vanish. There are also some triangulated categories defined over a field, the non-standard finite $1$-Calabi-Yau triangulated categories \cite{amiot_structure_2007}, for which no enhancements are known, but their triangulated structures are also defined from universal Massey products.

We concentrate in the simplest example considered in \cite{muro_triangulated_2007}, the category $\C T=\free{\mathbb Z/4}$ of finitely generated free $\mathbb Z/4$-modules with the identity suspension functor $\Sigma=\operatorname{id}_{\C T}$. This category is idempotent complete since all projective $\mathbb Z/4$-modules are free. We showed that this category has a triangulated structure where
\[\mathbb Z/4\stackrel{2}\longrightarrow \mathbb Z/4\stackrel{2}\longrightarrow\mathbb Z/4\stackrel{2}\longrightarrow\mathbb Z/4\]
is an exact triangle. We will regard this category as a $\mathbb Z$-linear category and show that there must be a non-vanishing topological obstruction. Let us first place this triangulated structure within the abelian group of stable Toda brackets, see Theorem \ref{Toda_brackets_and_cohomology}. 

\begin{proposition}
	If $\C T=\free{\mathbb Z/4}$ and $\Sigma\colon\C T\rightarrow \C T$ is the identity functor, then the previous triangulated structure is the only existing triangulated structure on the pair $(\C T,\Sigma)$ and corresponds under the bijection in Corollary \ref{Heller_graded} to the non-trivial element in
	\[\hh{0,-1}{\modulesfp{\C T_\Sigma},\ext_{\C T_{\Sigma}}^{3,*}}\cong\mathbb Z/2.\]
\end{proposition}

\begin{proof}
	Right modules over $\free{\mathbb Z/4}$ are the same as $\mathbb Z/4$-modules. The stable module category $\modulesst{\C T}=\modulesst{\mathbb Z/4}$ is the category $\modulesfp{\mathbb Z/2}$ of finite-dimensional $\mathbb Z/2$-vector spaces since any finitely generated $\mathbb Z/4$-module is a finite direct sum of copies of $\mathbb Z/4$ and $\mathbb Z/2$. Using the short exact sequence
	\[\mathbb Z/2\hookrightarrow\mathbb Z/4\twoheadrightarrow\mathbb Z/2 \]
	it is easy to see that the cosyzygy functor $S$ is the identity, as $\Sigma$.
	By Proposition \ref{Heller_ungraded}, 
	\[\begin{split}
	\hh{0}{\modulesfp{\C T},\ext_{\C T}^3(-,\Sigma^{-1})}&=\hh{0}{\modulesfp{\C T},\widehat{\ext}_{\C T_{\Sigma}}^{3,-1}}\\&\cong \hh{0}{\modulesst{\C T},\homst_{\C T}(\Sigma,S^{3})}
	\end{split}\] 
	is the group of natural transformations $\Sigma\rightarrow S^3$ in $\modulesst{\C T}$, i.e.~the endomorphisms of the identity functor in $\modulesfp{\mathbb Z/2}$, which is $\mathbb Z/2$. 
	
	Since $\Sigma$ is the identity, the bottom map in the exact sequence \eqref{exact_cohomology_sequence} is multiplication by $2$. We have just seen that the source (and target) of this map is isomorphic to $\mathbb Z/2$. Hence $i^*$	is an isomorphism
	\[\hh{0,-1}{\modulesfp{\C T_\Sigma},\ext_{\C T_{\Sigma}}^{3,*}}\cong \hh{0}{\modulesfp{\C T},\ext_{\C T}^3(-,\Sigma^{-1})}\]
	and all Toda brackets are stable, see Theorem \ref{Toda_brackets_and_cohomology}. 
	
	Any triangulated structure on $(\C T,\Sigma)$ must correspond to the non-trivial element since it must be a unit in $\hh{0,*}{\modulesfp{\C T_\Sigma},\widehat{\ext}_{\C T_{\Sigma}}^{\bullet,*}}$, which is non-trivial. In particular, the previous triangulated structure on this pair is unique.
\end{proof}

Our strategy to prove that one of the obstructions must be non-vanishing will be to show that the no-trivial element is not in the image of the edge morphism in Theorem \ref{theorem_edge_morphism_top},
\[\hbw{3,-1}{\C T_{\Sigma}, \C T_{\Sigma}}\To \hbw{0,-1}{\modulesfp{\C T_{\Sigma}},\ext_{\C T_{\Sigma}}^{3,\ast}}\cong\mathbb Z/2.
\]
This indeed guarantees that some obstriction is not zero, but it does not say which one. We could compute the obstructions explicitly but that would take much longer (computing spectral sequence differentials is difficult) and there would not be a clear benefit. The  computation goes through several steps, where we will use the following ungraded version of the spectral sequence in Proposition \ref{spectral_sequence_top}.

\begin{proposition}\label{spectral_sequence_ungraded}
	If $\C T$ is a small additive category such that $\modulesfp{\C T}$ is abelian, then there is a first quadrant cohomological spectral sequence 
	\[E_2^{p,q}=H^{p}(\modulesfp{\C T },\ext_{\C T }^{p})\Longrightarrow H^{p+q}(\C T ,\C T ).\]
\end{proposition}

\begin{proof}
	It is the spectral sequence of the bicomplex
	\[\hom_{\mathbb Z\modulesfp{\C T }^{\env}}^*(B_{\star}(\mathbb Z\modulesfp{\C T }),\hom_{\C T }^*(\mathbb Z(-)\otimes_{\mathbb Z\C T }B_{\bullet}(\mathbb Z\C T )\otimes_{\mathbb Z\C T }\C T ,-)).\]
	The proof is exactly the same as for Proposition \ref{spectral_sequence_top} since the hypotheses imply that any finitely presented right $\C T$-module has a projective resolution by objects in $\C T$.
\end{proof}

This proposition has a Hochschild analogue, but we will not use it here.

\begin{remark}\label{morphism_ss}
	It is straightforward to notice, looking at the bicomplexes defining these spectral sequences, that the following graded-to-ungraded comparison morphisms 
	\begin{gather*}
	i^*\colon H^{p,0}(\modulesfp{\C T_{\Sigma}},\ext_{\C T_{\Sigma}}^{p,*})\longrightarrow H^{p}(\modulesfp{\C T },\ext_{\C T }^{p}),\\
	i^*\colon H^{p+q,0}(\C T_{\Sigma},\C T_{\Sigma})\longrightarrow H^{p+q}(\C T ,\C T ),
	\end{gather*}
	which fit into the long exact sequence \eqref{exact_cohomology_sequence_top} derived from Proposition \ref{exact_cohomology_sequence}, 
	are part of a morphism from the spectral sequence in Proposition \ref{spectral_sequence_top} to that in Proposition \ref{spectral_sequence_ungraded}.	
\end{remark}

We will actually use a version with coefficients of the previous spectral sequence.

\begin{proposition}\label{spectral_sequence_with_coefficients}
	If $\C T$ is a small additive category such that $\modulesfp{\C T}$ is abelian and $M$ is a $\C T$-bimodule, there is a first quadrant cohomological spectral sequence of graded abelian groups
	\[E_2^{p,q}=H^{p}(\modulesfp{\C T },\ext_{\C T }^{p}(-,-\otimes_{\C T}M))\Longrightarrow H^{p+q}(\C T ,M ).\]
\end{proposition}

\begin{proof}
	It is the spectral sequence of the bicomplex
	\[\hom_{\mathbb Z\modulesfp{\C T }^{\env}}^*(B_{\star}(\mathbb Z\modulesfp{\C T }),\hom_{\C T }^*(\mathbb Z(-)\otimes_{\mathbb Z\C T }B_{\bullet}(\mathbb Z\C T )\otimes_{\mathbb Z\C T }\C T ,-\otimes_{\C{T}}M)).\]
	Observe that this is the same bicomplex as in the proof of Proposition  \ref{spectral_sequence_ungraded} with the last slot tensored by $M$. The proof is not any different to the proofs of Propositions \ref{spectral_sequence_top} and \ref{spectral_sequence_ungraded} since the last slot does not play any relevant role therein, it only slightly changes the $E_2$ term and the target.
\end{proof}

This proposition has graded and Hochschild analogues that will not be used in this paper. The particular version we present here is a special case of \cite[Theorem B]{jibladze_cohomology_1991}. The target in Jibladze and Pirashvili's spectral sequence is apparently different to ours, but it is possible to check that both coincide under our assumptions.


\begin{remark}\label{morphism_ss_coefficients}
	As in Remark \ref{morphism_ss}, the following graded-to-ungraded comparison morphisms 
	\begin{gather*}
	i^*\colon H^{p,r}(\modulesfp{\C T_{\Sigma}},\ext_{\C T_{\Sigma}}^{p,*})\longrightarrow H^{p}(\modulesfp{\C T },\ext_{\C T }^{p}(-,\Sigma^{r})),\\
	i^*\colon H^{p+q,r}(\C T_{\Sigma},\C T_{\Sigma})\longrightarrow H^{p+q}(\C T ,\C T(-,\Sigma^{r})),
	\end{gather*}
	also fitting into the long exact sequence \eqref{exact_cohomology_sequence_top} derived from Proposition \ref{exact_cohomology_sequence}, 
	are part of a morphism from the spectral sequence in Proposition \ref{spectral_sequence_top} to that in Proposition \ref{spectral_sequence_with_coefficients} for $M=\C T_{\Sigma}^r=\C T(-,\Sigma^r)$.
\end{remark}

We can define the \emph{Hochschild homology} $HH_\star(\C C, M)$ of a graded $k$-linear category $\C C$ with coefficients in a $\C C$-bimodule $M$ as the homology of $B_\ast(\C C)\otimes_{\C C^{\env}}M$. Here $\star$ is the Hochschild grading but there is an extra hidden grading coming from the fact that $\C C$ is graded, as with cohomology. Moreover, if $\C C$ is non-linear, i.e.~enriched in graded sets, we can define the \emph{homology} $H_\star(\C C,M)$ of $\C C$ with coefficients in a $\mathbb Z\C C$-bimodule $M$ as $HH_\star(\mathbb Z\C C, M)$. The ungraded cases are just  particular cases of this.

Let $R$ be a ring and $\free{R}$ the category of finitely generated free $R$-modules. The homology $H_\star(\free{R},\hom_R)$ is the Mac Lane or topological Hochschild homology of $R$ \cite{jibladze_cohomology_1991,pirashvili_mac_1992}. We need yet another spectral sequence, this time a universal coefficients one.

\begin{proposition}\label{universal_coefficients}
	For any commutative ring $R$ and any $R$-module $M$, there is a first quadrant cohomological spectral sequence
	\[E_2^{p,q}=\ext^p_R(H_q(\free{R},\hom_R),M)\Longrightarrow H^{p+q}(\free{R},\hom_R(-,-\otimes_RM)).\]
\end{proposition}

\begin{proof}
	Given two finitely generated free $R$-modules $P$ and $Q$, there are natural isomorphisms
	$$\hom_R(P,Q\otimes_R M)\st{\alpha}\longleftarrow \hom_R(P,Q)\otimes_RM\st{\beta}\To\hom_R(\hom_R(Q,P),M)$$
	defined by
	\begin{align*}
		\alpha(f\otimes m)(x)&=f(x)\otimes m\\
		\beta (f\otimes m)(g)&=\operatorname{trace}(fg)\cdot m.
	\end{align*}
	Indeed the morphisms $\alpha$ and $\beta$ are clearly well defined and natural, so it is enough to check that they are isomorphisms for $P=Q=R$, and this case is trivial.
	
	The isomorphisms $\alpha$ and $\beta$ define an isomorphism of cochain complexes
	\begin{multline*}
		 \hom_{\mathbb Z\free{R}^{\env}}(B_\star(\mathbb Z\free{R}),\hom_R(-,-\otimes_RM))\\
		 \cong \hom_R(B_\star(\mathbb Z\free{R})\otimes_{\mathbb Z\free{R}} \hom_R,M).
	\end{multline*}
	Indeed, dimension-wise these complexes look like
	\begin{gather*}
		\prod_{P_0,\dots,P_n}\hom_{\mathbb Z}\left(\bigotimes_{i=1}^n \mathbb Z\hom_R(P_i,P_{i-1}), \hom_R(P_n,P_0\otimes M)\right),\\
		\prod_{P_0,\dots,P_n}\hom_{R}\left(\bigotimes_{i=1}^n \mathbb Z\hom_R(P_i,P_{i-1})\otimes \hom_R(P_0,P_n), M\right),
	\end{gather*}
	where the $P_i$ run over all sequences of $n$ finitely generated free $R$-modules. Hence, the top module is a product of copies of $\hom_R(P_n,P_0\otimes M)$ and the bottom module, by adjunction, is a product of copies of $\hom_R(\hom_R(P_n,P_0),M)$. Both products are indexed by the same set, which is the union of all sets $\prod_{i=1}^n \hom_R(P_i,P_{i-1})$ over all possible choices of $P_0,\dots, P_n$, so we just apply the isomorphism $\beta\alpha^{-1}$ factorwise. Compatibility with differentials is straightforward.
	
	If $I^*$ is an injective resolution of $M$, then the spectral sequence of the statement is the first-horizontal-then-vertical cohomology spectral sequence of the bicomplex
	\[\hom_R(B_\star(\mathbb Z\free{R})\otimes_{\mathbb Z\free{R}} \hom_R,I^*).\]
	The identification of the $E_2$-term is obvious because the functors $\hom_R(-,I^n)$ are exact. For the target of this spectral sequence, we use the other one, obtained by first taking vertical homology and then horizontal homology. As we have seen above, $B_\star(\mathbb Z\free{R})\otimes_{\mathbb Z\free{R}} \hom_R$ consists of free $R$-modules, so the vertical homology of the bicomplex is $\hom_R(B_\star(\mathbb Z\free{R})\otimes_{\mathbb Z\free{R}} \hom_R,M)$ concentrated in vertical degree $0$. Hence we are done.
\end{proof}

\begin{corollary}
	For any $\mathbb Z/4$-module $M$ we have a natural isomorphism \[H^{3}(\free{\mathbb Z/4},\hom_{\mathbb Z/4}(-,-\otimes_{\mathbb Z/4}M))\cong \hom_{\mathbb Z/4}(\mathbb Z/2,M).\]
\end{corollary}

\begin{proof}
	The topological Hochschild homology of $\mathbb Z/4$ is computed in \cite{brun_topological_2000}. The lower homology groups are
	\begin{equation*}
	H_n(\free{\mathbb Z/4},\hom_{\mathbb Z/4}) \;\;\cong\;\;\left\{\begin{array}{ll}
	\mathbb Z/4,&n=0,\\
	0,&n=1,\\
	\mathbb Z/4,&n=2,\\
	\mathbb Z/2,&n=3.
	\end{array}\right.
	\end{equation*}
	This implies that the $E_2$-term of the universal coefficients spectral sequence in Proposition \ref{universal_coefficients} satisfies $E^{p,q}_2=0$ for $p>0$ and $q<3$. Therefore 
	\[
	H^{3}(\free{\mathbb Z/4},\hom_{\mathbb Z/4}(-,-\otimes_{\mathbb Z/4}M))\cong E_2^{0,3}=
	\hom_{\mathbb Z/4}(\mathbb Z/2,M).
	\]
\end{proof}

\begin{corollary}\label{trivial}
	The natural projection $\mathbb Z/4\twoheadrightarrow \mathbb Z/2$ induces the trivial morphism 
	\[H^{3}(\free{\mathbb Z/4},\hom_{\mathbb Z/4}) \stackrel{0}\longrightarrow
	H^{3}(\free{\mathbb Z/4},\hom_{\mathbb Z/4}(-,-\otimes_{\mathbb Z/4}\mathbb Z/2)).\]
\end{corollary}

\begin{proposition}\label{injective}
	The natural projection $\mathbb Z/4\twoheadrightarrow \mathbb Z/2$ induces an injective morphism
	\[H^0(\modulesfp{\mathbb Z/4},\ext^n_{\mathbb Z/4})\hookrightarrow
	H^0(\modulesfp{\mathbb Z/4},\ext^n_{\mathbb Z/4}(-,-\otimes_{\mathbb Z/4}\mathbb Z/2)).\]
\end{proposition}

\begin{proof}	
	An element $\varphi$ in the source is a family of elements
	\[\varphi(N)\in \ext^n_{\mathbb Z/4}(N,N),\]
	where $N$ runs over all finitely generated $\mathbb Z/4$-modules, and similarly for the target. Any $N$ decomposes as a direct sum of copies of $\mathbb Z/4$ and $\mathbb Z/2$. Hence, by \cite{baues_sum-normalised_1996}, $\varphi$ only depends on the two special values $\varphi(\mathbb Z/4)$ and $\varphi(\mathbb Z/2)$, and the former is zero because $\mathbb Z/4$ is projective, so the latter suffices. Now the result follows from the fact that the natural projection obviously induces an injective morphism (actually an isomorphism)
	\[\ext^n_{\mathbb Z/4}(\mathbb Z/2,\mathbb Z/2\otimes_{\mathbb Z/4} \mathbb Z/4)\longrightarrow \ext^n_{\mathbb Z/4}(\mathbb Z/2,\mathbb Z/2\otimes_{\mathbb Z/4} \mathbb Z/2).\]
\end{proof}

We finally prove the result which implies that not all our first obstructions vanish, and hence $\free{\mathbb Z/4}$ does not have any topological enhancement.

\begin{proposition}
	If $\C T=\free{\mathbb Z/4}$ and $\Sigma\colon\C T\rightarrow \C T$ is the identity functor, then the edge morphism 
	\[\hbw{3,-1}{\C T_{\Sigma}, \C T_{\Sigma}}\To \hbw{0,-1}{\modulesfp{\C T_{\Sigma}},\ext_{\C T_{\Sigma}}^{3,\ast}}
	\]
	of the spectral sequence in Proposition \ref{spectral_sequence_top} is trivial.
\end{proposition}

\begin{proof}
	Recall that $\C T$ is idempotent complete, as required by Proposition \ref{spectral_sequence_top}, since projective $\mathbb Z/4$-modules are free, and $\C T$-modules are the same as $\mathbb Z/4$-modules.
	
	Using the spectral sequence morphism in Remark \ref{morphism_ss_coefficients} we obtain a commutative diagram 
	\begin{center}
		\begin{tikzcd}
			\hbw{3,-1}{\C T_{\Sigma}, \C T_{\Sigma}}\arrow[r, "i^*"]\arrow[d]& \hbw{3}{\C T , \C T(-,\Sigma^{-1}) }\arrow[d]\\
			\hbw{0,-1}{\modulesfp{\C T_{\Sigma}},\ext_{\C T_{\Sigma}}^{3,\ast}}\arrow[r, hook, "i^*"]& \hbw{0}{\modulesfp{\C T },\ext_{\C T }^{3}(-,\Sigma^{-1})}
		\end{tikzcd}
	\end{center}
	where the vertical arrows are edge morphisms and the horizontal arrows are graded-to-ungraded comparison morphisms. The one at the bottom is injective by \eqref{exact_cohomology_sequence_top}. Therefore it suffices to show that the right edge morphism vanishes. The suspension functor is the identity, hence the right edge morphism is the morphism on the left of the following commutative diagram
	\begin{center}
		\begin{tikzcd}
			\hbw{3}{\free{\mathbb Z/4}, \hom_{\mathbb Z/4}}\arrow[d]\arrow[r,"0"]&\hbw{3}{\free{\mathbb Z/4}, \hom_{\mathbb Z/4}(-,-\otimes_{\mathbb Z/4}\mathbb Z/2)} \arrow[d]\\
			\hbw{0}{\modulesfp{\mathbb Z/4},\ext_{\mathbb Z/4}^{3}}
			\arrow[r,hook]& \hbw{0}{\modulesfp{\mathbb Z/4},\ext_{\mathbb Z/4}^{3}(-,-\otimes_{\mathbb Z/4}\mathbb Z/2)}
		\end{tikzcd}
	\end{center}
	This diagram is defined by the natural projection $\mathbb Z/4\twoheadrightarrow\mathbb Z/2$ and by the naturality in $M$ of the spectral sequence in Proposition \ref{spectral_sequence_with_coefficients}. The top arrow is trivial by Corollary \ref{trivial} and the bottom arrow is injective by Proposition \ref{injective}, therefore the left vertical arrow is necessarily trivial.

\end{proof}

We would like to remark that the computations carried out in this section are closely related to (some of them even directly taken from) the first (unpublished) proof of the fact that the triangulated structure of $\free{\mathbb Z/4}$ does not admit topological enhancements \cite{muro_triangulated_2007-1}.



\begin{thebibliography}{LVdB05}
	
	\bibitem[Ami07]{amiot_structure_2007}
	Claire Amiot, \emph{On the structure of triangulated categories with finitely
		many indecomposables}, Bull. Soc. Math. France \textbf{135} (2007), no.~3,
	435--474. \MR{2430189}
	
	\bibitem[And67]{andre_methode_1967}
	Michel André, \emph{Méthode simpliciale en algèbre homologique et algèbre
		commutative}, Lecture {Notes} in {Mathematics}, {Vol}. 32, Springer-Verlag,
	Berlin-New York, 1967. \MR{0214644}
	
	\bibitem[Ang08]{angeltveit_topological_2008}
	Vigleik Angeltveit, \emph{Topological {Hochschild} homology and cohomology of
		\textit{{A}}$_{\textrm{$\infty$}}$ ring spectra}, Geom. Topol. \textbf{12}
	(2008), no.~2, 987--1032. \MR{2403804}
	
	\bibitem[BD89]{baues_cohomology_1989}
	Hans~Joachim Baues and Winfried Dreckmann, \emph{The cohomology of homotopy
		categories and the general linear group}, \$K\$-Theory \textbf{3} (1989),
	no.~4, 307--338. \MR{1047191}
	
	\bibitem[BK91]{bondal_enhanced_1991}
	A.~I. Bondal and M.~M. Kapranov, \emph{Enhanced triangulated categories}, Mat.
	USSR Sb. \textbf{70} (1991), no.~1, 93--107.
	
	\bibitem[BKS04]{benson_realizability_2004}
	David Benson, Henning Krause, and Stefan Schwede, \emph{Realizability of
		modules over {Tate} cohomology}, Trans. Amer. Math. Soc. \textbf{356} (2004),
	no.~9, 3621--3668. \MR{2055748}
	
	\bibitem[BM07]{baues_homotopy_2007}
	Hans-Joachim Baues and Fernando Muro, \emph{The homotopy category of
		pseudofunctors and translation cohomology}, J. Pure Appl. Algebra
	\textbf{211} (2007), no.~3, 821--850. \MR{2344231}
	
	\bibitem[BM08]{baues_cohomologically_2008}
	H.-J. Baues and F.~Muro, \emph{Cohomologically triangulated categories. {I}},
	J. K-Theory \textbf{1} (2008), no.~1, 3--48. \MR{2424565}
	
	\bibitem[Bru00]{brun_topological_2000}
	M.~Brun, \emph{Topological {Hochschild} homology of {Z}/p$^{\textrm{n}}$}, J.
	Pure Appl. Algebra \textbf{148} (2000), no.~1, 29--76. \MR{1750729}
	
	\bibitem[BS01]{balmer_idempotent_2001}
	Paul Balmer and Marco Schlichting, \emph{Idempotent completion of triangulated
		categories}, J. Algebra \textbf{236} (2001), no.~2, 819--834. \MR{1813503}
	
	\bibitem[BT96]{baues_sum-normalised_1996}
	Hans-Joachim Baues and Andy Tonks, \emph{On sum-normalised cohomology of
		categories, twisted homotopy pairs and universal {Toda} brackets}, Quart. J.
	Math. Oxford Ser. (2) \textbf{47} (1996), no.~188, 405--433. \MR{1460232}
	
	\bibitem[CS15]{canonaco_uniqueness_2015}
	Alberto Canonaco and Paolo Stellari, \emph{Uniqueness of dg enhancements for
		the derived category of a {Grothendieck} category}, arXiv:1507.05509 [math]
	(2015), to appear in J. Eur. Math. Soc.
	
	\bibitem[Dim09]{dimitrova_triangulated_2009}
	Boryana Dimitrova, \emph{Triangulated structures for projective modules},
	arXiv:0912.4708 [math] (2009).
	
	\bibitem[Fre66]{freyd_stable_1966}
	Peter Freyd, \emph{Stable homotopy}, Proc. {Conf}. {Categorical} {Algebra}
	({La} {Jolla}, {Calif}., 1965), Springer, New York, 1966, pp.~121--172.
	\MR{0211399}
	
	\bibitem[Hel68]{heller_stable_1968}
	Alex Heller, \emph{Stable homotopy categories}, Bull. Amer. Math. Soc.
	\textbf{74} (1968), 28--63. \MR{0224090}
	
	\bibitem[Hov99]{hovey_model_1999}
	Mark Hovey, \emph{Model categories}, Mathematical {Surveys} and {Monographs},
	vol.~63, American Mathematical Society, Providence, RI, 1999. \MR{1650134}
	
	\bibitem[JP91]{jibladze_cohomology_1991}
	Mamuka Jibladze and Teimuraz Pirashvili, \emph{Cohomology of algebraic
		theories}, J. Algebra \textbf{137} (1991), no.~2, 253--296. \MR{1094244}
	
	\bibitem[Kad80]{kadeishvili_theory_1980}
	T.~V. Kadeishvili, \emph{On the theory of homology of fiber spaces}, Uspekhi
	Mat. Nauk \textbf{35} (1980), no.~3(213), 183--188, International Topology
	Conference (Moscow State Univ., Moscow, 1979).
	
	\bibitem[Kad82]{kadeishvili_algebraic_1982}
	\bysame, \emph{The algebraic structure in the homology of an
		\textit{{A}}($\infty$)-algebra}, Soobshch. Akad. Nauk Gruzin. SSR
	\textbf{108} (1982), no.~2, 249--252 (1983). \MR{720689}
	
	\bibitem[Kaj13]{kajiura_-enhancements_2013}
	Hiroshige Kajiura, \emph{On \textit{{A}}$_{\textrm{$\infty$}}$-enhancements for
		triangulated categories}, J. Pure Appl. Algebra \textbf{217} (2013), no.~8,
	1476--1503. \MR{3030547}
	
	\bibitem[Kel05]{kelly_basic_2005}
	G.~M. Kelly, \emph{Basic concepts of enriched category theory}, Repr. Theory
	Appl. Categ. (2005), no.~10, vi+137. \MR{2177301}
	
	\bibitem[LH03]{lefevre-hasegawa_sur_2003}
	Kenji Lefèvre-Hasegawa, \emph{Sur les
		\textit{{A}}$_{\textrm{$\infty$}}$-catégories}, Ph.D. thesis, Université
	Paris 7, 2003.
	
	\bibitem[LO10]{lunts_uniqueness_2010}
	Valery~A. Lunts and Dmitri~O. Orlov, \emph{Uniqueness of enhancement for
		triangulated categories}, J. Amer. Math. Soc. \textbf{23} (2010), no.~3,
	853--908. \MR{2629991}
	
	\bibitem[LVdB05]{lowen_hochschild_2005}
	Wendy Lowen and Michel Van~den Bergh, \emph{Hochschild cohomology of abelian
		categories and ringed spaces}, Adv. Math. \textbf{198} (2005), no.~1,
	172--221. \MR{2183254}
	
	\bibitem[Mit72]{mitchell_rings_1972}
	Barry Mitchell, \emph{Rings with several objects}, Advances in Math. \textbf{8}
	(1972), 1--161. \MR{0294454}
	
	\bibitem[ML63]{mac_lane_homology_1963}
	Saunders Mac~Lane, \emph{Homology}, Die {Grundlehren} der mathematischen
	{Wissenschaften}, {Bd}. 114, Academic Press, Inc., Publishers, New York;
	Springer-Verlag, Berlin-Göttingen-Heidelberg, 1963. \MR{0156879}
	
	\bibitem[MSS07]{muro_triangulated_2007}
	Fernando Muro, Stefan Schwede, and Neil Strickland, \emph{Triangulated
		categories without models}, Invent. Math. \textbf{170} (2007), no.~2,
	231--241. \MR{2342636}
	
	\bibitem[Mur06]{muro_functoriality_2006}
	Fernando Muro, \emph{On the functoriality of cohomology of categories}, J. Pure
	Appl. Algebra \textbf{204} (2006), no.~3, 455--472. \MR{2185612}
	
	\bibitem[Mur07]{muro_triangulated_2007-1}
	\bysame, \emph{A triangulated category without models}, arXiv:math/0703311
	(2007).
	
	\bibitem[Mur15]{muro_enhanced_2015}
	\bysame, \emph{Enhanced \textit{{A}}-infinity obstruction theory},
	arXiv:1510.00312 [math] (2015).
	
	\bibitem[Nee01]{neeman_triangulated_2001}
	Amnon Neeman, \emph{Triangulated categories}, Annals of {Mathematics}
	{Studies}, vol. 148, Princeton University Press, Princeton, NJ, 2001.
	\MR{1812507}
	
	\bibitem[Pup62]{puppe_formal_1962}
	D.~Puppe, \emph{On the formal structure of stable homotopy theory}, Colloquium
	on {Algebraic} {Topology}, Matematisk Institut, Aarhus Universitet, Aarhus,
	1962, pp.~65--71.
	
	\bibitem[PW92]{pirashvili_mac_1992}
	Teimuraz Pirashvili and Friedhelm Waldhausen, \emph{Mac {Lane} homology and
		topological {Hochschild} homology}, J. Pure Appl. Algebra \textbf{82} (1992),
	no.~1, 81--98. \MR{1181095}
	
	\bibitem[RB18]{rizzardo_k-linear_2018}
	Alice Rizzardo and Michel Van~den Bergh, \emph{A \textit{k}-linear triangulated
		category without a model}, arXiv:1801.06344 [math] (2018).
	
	\bibitem[Sch10]{schwede_algebraic_2010}
	Stefan Schwede, \emph{Algebraic versus topological triangulated categories},
	Triangulated categories, London {Math}. {Soc}. {Lecture} {Note} {Ser}., vol.
	375, Cambridge Univ. Press, Cambridge, 2010, pp.~389--407. \MR{2681714}
	
	\bibitem[Str69]{street_homotopy_1969}
	Ross Street, \emph{Homotopy classification of filtered complexes}, Ph.D.
	thesis, University of Sydney, 1969.
	
	\bibitem[Tab10]{tabuada_matrix_2010}
	Gonçalo Tabuada, \emph{Matrix invariants of spectral categories}, Int. Math.
	Res. Not. IMRN (2010), no.~13, 2459--2511. \MR{2669656}
	
	\bibitem[Ver96]{verdier_categories_1996}
	Jean-Louis Verdier, \emph{Des catégories dérivées des catégories
		abéliennes}, Astérisque (1996), no.~239, xii+253 pp. (1997). \MR{1453167}
	
\end{thebibliography}

\providecommand{\bysame}{\leavevmode\hbox to3em{\hrulefill}\thinspace}
\providecommand{\MR}{\relax\ifhmode\unskip\space\fi MR }
\providecommand{\MRhref}[2]{%
	\href{http://www.ams.org/mathscinet-getitem?mr=#1}{#2}
}
\providecommand{\href}[2]{#2}

\end{document}